\documentclass[10pt]{amsart}
  
\usepackage[T1]{fontenc}
\usepackage[small,it]{caption}
\usepackage[english]{babel}
\usepackage{url}
\usepackage{fullpage}
\usepackage{graphicx}                         
\usepackage{amsmath}
\usepackage{amsthm}                      
\usepackage{amssymb} 
\usepackage{mathrsfs}
\usepackage{stmaryrd}
\usepackage{enumerate}
\usepackage{tikz}
\usepackage{needspace}

\usepackage{tikz-cd}
\usetikzlibrary{decorations.pathmorphing}

\DeclareFontFamily{U}{wncy}{}
\DeclareFontShape{U}{wncy}{m}{n}{<->wncyr10}{}
\DeclareSymbolFont{mcy}{U}{wncy}{m}{n}
\DeclareMathSymbol{\sha}{\mathord}{mcy}{"58}
\usepackage{euscript}
\renewcommand{\mathcal}{\EuScript}

\swapnumbers
\theoremstyle{plain}
\makeatletter
\def\swappedhead#1#2#3{%
	\thmnumber{\@upn{\the\thm@headfont#2\@ifnotempty{#1}{.~}}}%
	\thmname{#1}%
	\thmnote{ {\the\thm@notefont(#3)}}}
\makeatother

\usepackage{color}
\usepackage{hyperref}
\usepackage[noabbrev,capitalize]{cleveref}
\hypersetup{
	colorlinks,
	linkcolor={red!50!black},
	citecolor={blue!50!black},
	urlcolor={blue!80!black}
}
\usepackage{palatino}
\usepackage{mathtools}

\newtheorem{thm}{Theorem}[section]
\newtheorem{thmA}[thm]{Theorem A}
\newtheorem{thmB}[thm]{Theorem B}

\newtheorem{lem}[thm]{Lemma}
\newtheorem{prop}[thm]{Proposition}
\newtheorem{cor}[thm]{Corollary}

\theoremstyle{definition}
\newtheorem{example}[thm]{Example}
\newtheorem{rem}[thm]{Remark}
\newtheorem{defn}[thm]{Definition}

\newtheorem{para}[thm]{}

\newcommand{\K}{\mathbb{K}}
\newcommand{\Z}{\mathbf{Z}}
\newcommand{\rats}{\mathbf{Q}}

\renewcommand{\P}{\mathsf{P}}
\newcommand{\Q}{\mathsf{Q}}
\newcommand{\C}{\mathsf{C}}
\newcommand{\D}{\mathsf{D}}
\newcommand{\M}{\mathsf{M}}
\newcommand{\N}{\mathsf{N}}
\newcommand{\com}{\mathsf{Com}}
\newcommand{\ass}{\mathsf{Ass}}
\newcommand{\lie}{\mathsf{Lie}}

\newcommand{\coass}{\mathsf{coAss}}
\newcommand{\colie}{\mathsf{coLie}}
\newcommand{\End}{\mathsf{End}}
\newcommand{\coEnd}{\mathsf{coEnd}}

\newcommand{\Aoo}{\mathsf{A}_\infty}
\newcommand{\Coo}{\mathsf{C}_\infty}
\newcommand{\Loo}{\mathsf{L}_\infty}

\newcommand{\susp}{S^{-1}} %This is not exactly the S I was looking for. Can someone find the same they use in LV?

\newcommand{\Palg}{\mathsf{P}\text{-}\mathsf{alg}}

\newcommand{\Picoalg}{\mathsf{P}^\antishriek\text{-}\mathsf{coalg}}

\newcommand{\Qalg}{\mathsf{Q}\text{-}\mathsf{alg}}
\newcommand{\Qicoalg}{\mathsf{Q}^\antishriek\text{-}\mathsf{coalg}}

\newcommand{\BarLC}{\mathcal{C}}
\newcommand{\CobarLC}{\mathcal{L}}
\newcommand{\BarAC}{\mathrm{B}}
\newcommand{\CobarAC}{\Omega}
\newcommand{\sym}{\mathrm{Sym}}
\newcommand{\Def}{\mathrm{Def}}
\newcommand{\coDef}{\mathrm{coDef}}

\newcommand{\F}{{F}}
\newcommand{\Gr}{\operatorname{Gr}}

\newcommand{\MC}{\mathrm{MC}}

\makeatletter
\newcommand*{\doublerightarrow}[2]{\mathrel{
		\settowidth{\@tempdima}{$\scriptstyle#1$}
		\settowidth{\@tempdimb}{$\scriptstyle#2$}
		\ifdim\@tempdimb>\@tempdima \@tempdima=\@tempdimb\fi
		\mathop{\vcenter{
				\offinterlineskip\ialign{\hbox to\dimexpr\@tempdima+1em{##}\cr
					\rightarrowfill\cr\noalign{\kern.5ex}
					\rightarrowfill\cr}}}\limits^{\!#1}_{\!#2}}}
\newcommand*{\triplerightarrow}[1]{\mathrel{
		\settowidth{\@tempdima}{$\scriptstyle#1$}
		\mathop{\vcenter{
				\offinterlineskip\ialign{\hbox to\dimexpr\@tempdima+1em{##}\cr
					\rightarrowfill\cr\noalign{\kern.5ex}
					\rightarrowfill\cr\noalign{\kern.5ex}
					\rightarrowfill\cr}}}\limits^{\!#1}}}
\makeatother

\newcommand{\QQ}{\mathrm Q}

\title{Lie, associative and commutative quasi-isomorphism}

\author{Ricardo Campos}
\address{Ricardo Campos\\Institut de Mathématiques de Toulouse, UMR5219, Université de Toulouse, CNRS, UPS, F-31062 Toulouse Cedex 9, France}
\email{ricardo.campos@math.univ-toulouse.fr}
\author{Dan Petersen}
\address{Dan Petersen\\{Matematiska Institutionen, Stockholms Universitet, 106 91 Stockholm, Sweden}}
\email{dan.petersen@math.su.se}
\author{Daniel Robert-Nicoud}
\address{Daniel Robert-Nicoud}
\email{daniel.robertnicoud@gmail.com}
\author{Felix Wierstra}
\address{Felix Wierstra\\{Korteweg-de Vries Institute for Mathematics,
University of Amsterdam, Science Park 105-107, 1098 XG Amsterdam, the Netherlands}}
\email{felix.wierstra@gmail.com}



\thanks{}

\begin{document}

\begin{abstract}
	Over a field of characteristic zero, we show that two commutative differential graded (dg) algebras are quasi-isomorphic if and only if they are quasi-isomorphic as associative dg algebras. This answers a folklore problem in rational homotopy theory, showing that the rational homotopy type of a space is determined by its associative dg algebra of rational cochains. We also show a Koszul dual statement, under an additional completeness hypothesis: two homotopy complete dg Lie algebras whose universal enveloping algebras are quasi-isomorphic as associative dg algebras must themselves be quasi-isomorphic. The latter result applies in particular to nilpotent Lie algebras (not differential graded), in which case it says that two nilpotent Lie algebras whose universal enveloping algebras are isomorphic as associative algebras must be isomorphic. 
\end{abstract}

\maketitle

\setcounter{tocdepth}{1}
\tableofcontents

\newcommand{\subsec}[1]{\needspace{1\baselineskip}{\begin{center}\textbf{{#1}}\end{center}}\nopagebreak}

\newcommand{\hc}[1]{{#1}^{h\wedge}}
\newcommand{\g}{\mathfrak g}
\newcommand{\antishriek}{\text{\raisebox{\depth}{\textexclamdown}}}
\newcommand{\h}{\mathfrak h}
\renewcommand{\hom}{\mathrm{Hom}}

\setcounter{section}{-1}

\section{Introduction}

\begin{para}
	Can one recover a Lie algebra $\g$ from its universal enveloping algebra $U\g$? Over a field of characteristic zero, one possible answer is \emph{yes}, as follows: the set of primitive elements in any bialgebra form a Lie algebra, and a consequence of the Poincar\'e--Birkhoff--Witt theorem is that the Lie algebra of primitive elements in $U\g$ is isomorphic to $\g$. However, note that this answer to the question assumes that we know the bialgebra structure of $U\g$. Let us suppose that we are only given $U\g$ as an associative algebra --- is it still possible to recover the Lie algebra $\g$? This question is in fact an open problem, which seems to have been first stated in print by Bergman \cite[p.\ 187]{bergman}. The results of this paper imply a positive answer to the question for \emph{nilpotent} Lie algebras: a nilpotent Lie algebra $\g$ is completely determined up to isomorphism by the associative algebra $U\g$. Perhaps surprisingly, the proof of this very concrete result will require passing through a study of the abstract deformation theory of $\infty$-algebras over operads. Our original motivation was a seemingly unrelated question arising from rational homotopy theory.
\end{para}

\begin{para}
	Recall that two (commutative, associative, Lie, ...) differential graded algebras $A$ and $B$ are said to be \emph{quasi-isomorphic} if they can be linked by a zig-zag
	\[
	A\stackrel{\sim}{\longleftarrow}\bullet\stackrel{\sim}{\longrightarrow}\cdots\stackrel{\sim}{\longleftarrow}\bullet\stackrel{\sim}{\longrightarrow}B
	\]
	of morphisms of (commutative, associative, Lie, ...) algebras, each of which induces an isomorphism on homology. 
\end{para}

\begin{para}
	A commutative dg algebra is in particular an associative dg algebra. This means that there are two a priori different notions of what it means for two commutative dg algebras to be quasi-isomorphic, since there are many more potential zig-zags in the larger category of associative dg algebras. One is led to ask: if two commutative dg algebras are quasi-isomorphic as associative dg algebras, must they be quasi-isomorphic also as commutative dg algebras? This turns out to be a surprisingly subtle question. Our first main theorem settles the question completely in characteristic zero.
\end{para}

\begin{thmA}\label{thmA intro}
	Let $A$ and $B$ be two commutative dg algebras over a field of characteristic zero. Then, $A$ and $B$ are quasi-isomorphic as  associative dg algebras if and only if they are also quasi-isomorphic as commutative dg algebras. 
\end{thmA}

\begin{para}
	Our second main theorem is \emph{Koszul dual} to Theorem \hyperref[thmA intro]{A}, informally speaking. The Koszul dual of a commutative dg algebra is a dg Lie algebra, and vice versa, and the Koszul dual of an associative dg algebra is an associative dg algebra. Moreover, Koszul duality interchanges the forgetful functor from commutative dg algebras to associative dg algebras and the universal enveloping functor from dg Lie algebras to associative dg algebras. Thus, one might expect a Koszul dual form of Theorem \hyperref[thmA intro]{A} to assert that if two dg Lie algebras have quasi-isomorphic universal enveloping algebras, then the dg Lie algebras are themselves quasi-isomorphic. Unfortunately we are not able to prove this statement (and it is not clear if one should expect it to be true), since Koszul duality is not an equivalence; some information is lost when passing from one side of the Koszul duality correspondence to the other.\footnote{However, by working with a different equivalence relation on coalgebras one can obtain a version of Koszul duality which is a genuine equivalence of categories --- see \cref{remark: general koszul duality}.} In order to carry out the proof we need to restrict our attention to a certain subcategory of homotopically complete algebras. This completeness hypothesis is also what one might expect based on recent work of Heuts \cite{heuts}, although our work is independent of his.
\end{para}

\begin{para}
	In order to state the next theorem we need to recall the notion of \emph{homotopy completion}, introduced by Harper--Hess \cite{harperhess} (in a far more general setting than the one considered here). Let $\g$ be a dg Lie algebra over a field of characteristic zero, with lower central series filtration
	\[
	\g = L^1\g \supseteq L^2 \g \supseteq L^3 \g \supseteq \ldots
	\]
	We define the \emph{completion} of $\g$ as the inverse limit $\g^\wedge \coloneqq \varprojlim \g/L^n\g$. The \emph{homotopy completion} of $\g$, denoted $\hc \g$, is defined to be the completion of a cofibrant replacement of $\g$, for example the bar-cobar resolution of $\g$. Any quasi-isomorphism between cofibrant dg Lie algebras induces a quasi-isomorphism between completions, so the homotopy completion is well defined up to quasi-isomorphism. 
\end{para}

\begin{thmB}\label{thm B}
	Let $\g$ and $\h$ be dg Lie algebras over a field of characteristic zero. If their universal enveloping algebras $U\g$ and $U\h$ are quasi-isomorphic as associative dg algebras then the homotopy completions $\hc\g$ and $\hc\h$ are quasi-isomorphic as dg Lie algebras.
\end{thmB}

\begin{para}\label{harper-hess-criterion}
	Theorem \hyperref[thm B]{B} raises a further question. Let us say that a dg Lie algebra $\g$ is \emph{homotopy complete} if $\g$ and $\hc\g$ are quasi-isomorphic. Clearly, Theorem \hyperref[thm B]{B} becomes more useful if $\g$ and $\h$ are known to be homotopy complete, and it is natural to ask for simple conditions ensuring this. Theorem 1.12(a) of \cite{harperhess}, specialized to our situation, says that a dg Lie algebra concentrated in positive homological degrees is homotopy complete. We prove the following result. 
\end{para}

\begin{thm}\label{homotopy-complete}
	Let $\g$ be a dg Lie algebra over a field of characteristic zero. Suppose that one of the following two conditions holds:
	\begin{itemize}
		\item[i)] $\g$ is nilpotent and concentrated in nonnegative homological degree.
		\item[ii)] $\g$ is concentrated in strictly negative homological degree.
	\end{itemize}
	Then $\g$ is homotopy complete. 
\end{thm}

\begin{para}
	Here, a dg Lie algebra is said to be \emph{nilpotent} if the lower central series is bounded below in each homological degree. In particular, any dg Lie algebra concentrated in positive degrees is nilpotent, so \cref{homotopy-complete} recovers the result of Harper--Hess mentioned in \S\ref{harper-hess-criterion}. \end{para}

\begin{para}
	Theorem \hyperref[thmA intro]{A} is trivial in case the algebras $A$ and $B$ have no differential, which is not the case for Theorem \hyperref[thm B]{B}. In fact, Theorem \hyperref[thm B]{B} is highly nontrivial even in the case where $\g$ and $\h$ are classical Lie algebras concentrated in degree $0$. It is particularly interesting when $\g$ and $\h$ are additionally assumed to be nilpotent, since in this case \cref{homotopy-complete} ensures that $\g$ and $\h$ are homotopy complete. In this case Theorem \hyperref[thm B]{B} says the following. 
\end{para}

\begin{cor}\label{classical Thm B}
	Let $\g$ and $\h$ be nilpotent Lie algebras over a field of characteristic zero. The universal enveloping algebras $U\g$ and $U\h$ are isomorphic as associative algebras if and only if $\g$ and $\h$ are isomorphic as Lie algebras.
\end{cor}

\begin{rem}
	According to a result of Riley--Usefi \cite{rileyusefi} it is known that if $\g$ and $\h$ are Lie algebras such that $U\g \cong U\h$ as associative algebras, then $\g$ is nilpotent if and only if $\h$ is nilpotent. 
\end{rem}

\begin{para}
	\cref{classical Thm B} makes progress on the long-standing problem of whether a Lie algebra can be recovered from its universal enveloping algebra {(seen as an associative algebra only)} \cite[p.\ 187]{bergman}. Before this, the statement was known only for some special cases. For example, Schneider--Usefi \cite{schneiderusefi} have a computer-assisted proof of the claim for all nilpotent Lie algebras of dimension at most $6$. Over a field of positive characteristic it is possible for $U\g$ and $U\h$  to be isomorphic, even as Hopf algebras, without $\g$ and $\h$ being isomorphic. We refer the reader to the survey paper \cite{usefi} for more information.
\end{para}

\begin{para}
	The question of whether a Lie algebra can be recovered from its universal enveloping algebra is analogous to the more well-studied question of whether a discrete group can be recovered from its group algebra (considered as an associative algebra); the latter problem was famously settled by Hertweck's construction \cite{hertweck} of two non-isomorphic finite groups $G$ and $H$ such that $\Z G \cong \Z H$. For finite nilpotent groups $G$ and $H$ it \emph{is} true that $\Z G \cong \Z H$ implies $G \cong H$ \cite{roggenkampscott}, which suggests that Theorem \hyperref[thm B]{B} might be ``sharp'' in the sense that some (pro)nilpotence condition is necessary in order to recover a (dg) Lie algebra from its universal enveloping algebra. 
\end{para}

\begin{para}
	Theorem \hyperref[thm B]{B} gives an interesting example where generalizing a problem makes it easier. Theorem \hyperref[thm B]{B} is a significantly stronger result than its special case, \cref{classical Thm B}. One could ask whether our methods could be simplified if we only wanted to give a proof of \cref{classical Thm B}, so that we could give a more direct argument in this special case. We do not believe that this is possible. Indeed, the very first step of our argument is to pass to the Koszul dual setting by applying the bar construction (also known as the functor of Chevalley--Eilenberg chains), so that even if we start with a classical (non-dg) Lie algebra, we immediately obtain something differential graded. The Koszul duality which is crucial for our arguments only makes sense in the differential graded world.
\end{para}

\begin{para}
	Theorem \hyperref[thmA intro]{A} gives a positive answer to a folklore problem in rational homotopy theory. Let $C^\ast(X,\Z)$ denote the singular cochains of a topological space $X$. It is well known that $C^\ast(X,\Z)$ is an associative dg algebra which is not commutative in general; the best one can say is that it admits the structure of an $\mathsf E_\infty$-algebra (an algebra which is commutative up to coherent higher homotopy). This $\mathsf E_\infty$-algebra structure is not in general quasi-isomorphic to a strictly commutative multiplication, as one can see from the non-triviality of cohomology operations like the Steenrod squares. \emph{Rationally}, however, every $\mathsf E_\infty$-algebra is quasi-isomorphic to a strictly commutative dg algebra, and Sullivan \cite{sullivaninfinitesimal} constructed a functor $A_{\mathrm{PL}}$ from spaces to commutative dg algebras over $\rats$ such that $C^\ast(X,\rats)$ is naturally quasi-isomorphic to $A_{\mathrm{PL}}(X)$. Sullivan also showed that if $X$ and $Y$ are nilpotent spaces of finite type, then $X$ and $Y$ have the same rational homotopy type if and only if $A_{\mathrm{PL}}(X)$ and $A_{\mathrm{PL}}(Y)$ are quasi-isomorphic as \emph{commutative} dg algebras. It is then natural to ask whether one can detect the rational homotopy type of $X$ using only the dg algebra $C^\ast(X,\rats)$, i.e.\ without invoking the functor $A_{\mathrm{PL}}$ --- a priori, an associative quasi-isomorphism between $C^\ast(X,\rats )$ and $C^\ast(Y,\rats)$ does not imply the existence of a commutative quasi-isomorphism between $A_{\mathrm{PL}}(X)$ and $A_{\mathrm{PL}}(Y)$. Theorem \hyperref[thmA intro]{A} gives a positive answer to the question.
\end{para}

\begin{cor}
	Let $X$ and $Y$ be connected, nilpotent, based spaces of finite $\rats$-type. Then $X$ and $Y$ are rationally homotopy equivalent if and only if the cochain algebras $C^\ast(X,\rats)$ and $C^\ast(Y,\rats)$ are quasi-isomorphic. 
\end{cor}

\begin{para}
	Theorem \hyperref[thm B]{B} also admits an immediate interpretation in rational homotopy theory, via Quillen's approach to rational homotopy theory using dg Lie algebras \cite{quillenrationalhomotopytheory} (which in fact predates Sullivan's work). Quillen constructed a functor $\lambda$ from based simply connected spaces to dg Lie algebras over $\rats$ such that $X$ and $Y$ are rationally homotopy equivalent if and only if $\lambda X$ and $\lambda Y$ are quasi-isomorphic, and such that there is a quasi-isomorphism of associative dg algebras between $U(\lambda X)$ and $C_\ast(\Omega X,\rats)$. The multiplication on $C_\ast(\Omega X,\rats)$ is given by the Pontryagin product, i.e.\ the product coming from the concatenation of loops, and we take $\Omega X$ to be the Moore loop space of the based space $X$, which has a strictly associative multiplication. Moreover, the dg Lie algebra $\lambda X$ is concentrated in positive degrees if $X$ is simply connected, so it is in particular homotopy complete. Theorem \hyperref[thm B]{B} implies the following statement in this context.
\end{para}

\begin{cor}\label{rht-cor}
	Let $X$ and $Y$ be simply connected based spaces. Then $X$ and $Y$ are rationally homotopy equivalent if and only if the algebras $C_\ast(\Omega X,\rats)$ and $C_\ast (\Omega Y,\rats)$ of chains on their Moore loop spaces are quasi-isomorphic as associative dg algebras. 
\end{cor}

\begin{rem}Quillen's original results for simply connected spaces were later extended also to finite type connected nilpotent spaces \cite{neisendorfer}, and the analogue of \cref{rht-cor} remains true if instead of assuming $X$ and $Y$ to be simply connected we suppose they are finite type connected nilpotent. 
\end{rem}

\begin{para}\label{saleh}
	In a paper that had a very strong influence on the present project, Saleh \cite{saleh} proved that a commutative dg algebra is formal as a dg algebra if and only if it is formal as a commutative dg algebra, and a dg Lie algebra is formal if and only if its universal enveloping algebra is formal as a dg associative algebra. These are special cases of our results, since formality says precisely that an algebra and its homology are quasi-isomorphic. However, one should note that Saleh's paper does not require any nilpotence or connectivity assumptions on the dg Lie algebras, so in this respect his result is stronger than ours. The starting point of the present paper was an attempt to see how far the arguments of Saleh could be pushed. Further prior work in the direction of Theorem Theorem \hyperref[thmA intro]{A} can be found in an answer \cite{mathoverflow} by Tyler Lawson to a question on MathOverflow, explaining an argument as to why Theorem \hyperref[thmA intro]{A} holds in the non-negatively graded case. 
\end{para}

%\begin{para}
%	In general very little seems to be known about the natural functor from the homotopy category of commutative dg algebras to the homotopy category of all dg algebras. It is not full: an easy example is given in \cite[p.1]{amrani}, where it is attributed to Lurie. To our knowledge it is unknown whether the functor is faithful; a partial result in this direction is given by Amrani \cite{amrani}. Our Theorem Theorem \hyperref[thmA intro]{A} says precisely that the functor is injective on isomorphism classes. 
%\end{para}
\begin{para}
	In a sense, our Theorem \hyperref[thmA intro]{A} is about probing how the forgetful functor from the homotopy category of commutative dg algebras to the homotopy category of all dg algebras fails to be \emph{fully faithful} --- as opposed to the classical (non-dg) forgetful functor from commutative rings to noncommutative rings, which is fully faithful. Our Theorem \hyperref[thmA intro]{A} shows that in characteristic zero, some shadow of fully faithfulness is mysteriously restored, in that the functor is injective on isomorphism classes. In a sequel to this paper \cite{followup} we show by similar methods that the same functor is in addition \emph{faithful}, also in characteristic zero, extending prior work of Amrani \cite{amrani}. However, the forgetful functor is certainly not full \cite[\S1.4]{followup}.
	\end{para}

\begin{para}
	In this paper we systematically use the language of operads, operadic algebras, and the Koszul duality theory of operads; the results are obtained by studying the interplay between the operads $\mathsf{Lie}$, $\mathsf{Ass}$ and $\mathsf{Com}$. In fact, the only property of these operads that we end up using (besides their Koszulness) is that the natural morphism $\mathsf{Lie} \to \mathsf{Ass}$ admits a left inverse in the category of infinitesimal bimodules over the operad $\mathsf{Lie}$. One obtains versions of Theorems \hyperref[thmA intro]{A} and \hyperref[thm B]{B} for any morphism of Koszul operads $\P \to \Q$ which is a split injection of infinitesimal $\P$-bimodules:
\end{para}

\begin{thm}\label{generalization}
	Let $f \colon \P \to \Q$ be a morphism of binary Koszul operads in characteristic zero with $\P(n)$ and $\Q(n)$ finite dimensional for all $n$. Let $\Q^! \to \P^!$ be the induced morphism between the Koszul dual operads. Suppose that there exists a morphism of infinitesimal $\P$-bimodules $s \colon \Q \to \P$ such that $s \circ f = \mathrm{id}_\P$. Then: 
	\begin{enumerate}[(i)]
		\item Two dg $\P^!$-algebras $A$ and $B$ are quasi-isomorphic if and only if they are quasi-isomorphic as $\Q^!$-algebras. 
		\item Let $A$ and $B$ be dg $\P$-algebras. If their derived operadic pushforwards $\mathbf Lf_!A$ and $\mathbf Lf_!B$ are quasi-isomorphic as dg $\Q$-algebras, then the homotopy completions $\hc A$ and $\hc {B}$ are quasi-isomorphic.
	\end{enumerate}
\end{thm}
	
\begin{para}
	By specializing \cref{generalization} to the case $\P=\lie$ and $\Q=\ass$, one recovers (more or less) Theorems \hyperref[thmA intro]{A} and \hyperref[thm B]{B}. Two technical remarks are in order:
	\begin{itemize}
		\item \cref{generalization} only considers binary quadratic operads. In particular, the corresponding algebras do not have units. In the body of the paper we prove versions of Theorems \hyperref[thmA intro]{A} and \hyperref[thm B]{B} that apply to unital algebras as well. The additional complications arising from the presence of units are treated by ad hoc arguments (\S \ref{lurie1}--\S\ref{lurie2} and \S\ref{different augmentations}--\S\ref{lemma:UL qi as unital algebras then also as augmented}) which do not apply to the case of general operadic algebras.
		\item In the statement of Theorem \hyperref[thm B]{B} we considered the usual universal enveloping algebra functor, but \cref{generalization}(ii) considers the derived version of the universal enveloping algebra \cite[\S 4.6]{hinichmodelstructure}. Nevertheless, \cref{generalization}(ii) specializes to Theorem \hyperref[thm B]{B}: the universal enveloping algebra functor always preserves quasi-isomorphisms (\cref{lemma:U preserves qi}), so in this case we have $f_!A \simeq f_!A'$ if and only if $\mathbf{L}f_!A \simeq \mathbf Lf_!A'$.
	\end{itemize}
	The proof of \cref{generalization} is a modification of the arguments given in the body of the paper; no part of it should be difficult for the reader comfortable with the necessary operadic formalism. We leave the details to the interested reader.
\end{para}
	
\begin{para}
	We know of one further example to which the general \cref{generalization} applies. By \cite[Section 6.2]{griffincomodules}, the morphism $\mathsf{Leib} \to \mathsf{Diass}$ from the Leibniz operad to the diassociative operad admits a left inverse as infinitesimal bimodule. It follows that two dg Zinbiel algebras are quasi-isomorphic if and only if they are quasi-isomorphic as dendriform algebras (the analogue of Theorem \hyperref[thmA intro]{A}), and two dg Leibniz algebras whose universal enveloping diassociative algebras are quasi-isomorphic must have quasi-isomorphic homotopy completions (the analogue of Theorem \hyperref[thm B]{B}).
\end{para}

\subsec{Outline of the arguments and structure of the paper}

\begin{para}
	The Koszul duality between Theorems \hyperref[thmA intro]{A} and \hyperref[thm B]{B} is clearly visible in the structure of the proofs. We will prove Theorem \hyperref[thmA intro]{A} by a direct argument, and Theorem \hyperref[thm B]{B} by dualizing to reduce to Theorem \hyperref[thmA intro]{A}, or rather to a statement very close to it.
\end{para}

\begin{para}\label{summary-1}
	Let us briefly summarize the proofs, focusing first on Theorem \hyperref[thmA intro]{A}. We will need to work with $\Aoo$-algebras rather than associative algebras, and similarly we will have to replace commutative algebras with $\Coo$-algebras (which are sometimes called ``commutative $\Aoo$-algebras'' in the older literature). The statement we will actually prove is that if two $\Coo$-algebras are $\Aoo$-quasi-isomorphic, then they are also $\Coo$-{quasi-}isomorphic. We will represent our two $\Coo$-algebra structures by two Maurer--Cartan elements of a certain dg Lie algebra $\mathfrak h$, called the \emph{deformation complex} of $\Coo$-algebra structures. Two Maurer--Cartan elements of the deformation complex are gauge equivalent if and only if the two $\Coo$-algebra structures are $\Coo$-quasi-isomorphic (in fact, $\Coo$-isotopic). The fact that they are $\Aoo$-quasi-isomorphic translates into the assertion that these two Maurer--Cartan elements are gauge equivalent in a larger dg Lie algebra $\mathfrak g$, which is the deformation complex of $\Aoo$-algebra structures. These dg Lie algebras are essentially the Harrison and Hochschild cochain complexes, respectively. One can now ask the following rather general question: consider complete filtered dg Lie algebras $\mathfrak h \subset \mathfrak g$, and suppose we are given two Maurer--Cartan elements in $\mathfrak h$ which are gauge equivalent in $\mathfrak g$. When are they also gauge equivalent in $\mathfrak h$? 
\end{para}
	
\begin{para}\label{summary-2}
	In \cref{section1}, we will give an answer to this more general question: this holds whenever there exists a filtered retraction of $\mathfrak g$ onto $\mathfrak h$ as an $\mathfrak h$-module. Thus, our goal becomes to construct a retraction of the Hochschild cochains onto the Harrison cochains. The existence of such a retraction goes back to Barr, but we will give a slightly different proof of this fact. In \cref{section2}, we observe that there is a retraction of the operad $\ass$ onto the operad $\lie$ as an infinitesimal bimodule over the operad $\lie$, as a consequence of the Poincar\'e--Birkhoff--Witt theorem. This implies in particular the existence of a filtered retraction of the Hochschild cochains onto the Harrison cochains. In \cref{sect:thma}, we put these ingredients together to prove Theorem \hyperref[thmA intro]{A}. 
\end{para}	
	
\begin{para}
	For the proof of Theorem \hyperref[thm B]{B} we will use Quillen's bar-cobar adjunction $\BarLC \vdash \CobarLC$ between dg Lie algebras and cocommutative dg coalgebras. It is well known that for any dg Lie algebra $\g$ there is a quasi-isomorphism of coassociative dg coalgebras between $\BarLC \g$ and $\BarAC U\g$, where $\BarAC$ denotes the classical bar construction of associative dg algebras. It follows that if $U\g$ and $U\h$ are quasi-isomorphic, then so are $\BarAC U\g$ and $\BarAC U\h$, which implies that $\BarLC \g$ and $\BarLC \h$ are quasi-isomorphic as \emph{coassociative} dg coalgebras, so that there is an $\Aoo$-quasi-isomorphism $\BarLC \g \rightsquigarrow \BarLC \h$. We will then prove a dual form of Theorem \hyperref[thmA intro]{A}, showing that there is in fact a $\Coo$-quasi-isomorphism $\BarLC \g \rightsquigarrow \BarLC \h$. At this point one might hope to apply the cobar functor to deduce that $\g \simeq \h$. Unfortunately this is problematic, for two reasons:
	\begin{itemize}
		\item The cobar functor $\CobarLC$ does not preserve quasi-isomorphisms in general.
		\item In general, a $\Coo$-morphism between cocommutative dg coalgebras does not induce a morphism between their cobar constructions, unless some finiteness conditions are imposed.
	\end{itemize}
	It turns out that both problems are solved by replacing $\CobarLC$ with the \emph{completed} cobar functor $\CobarLC^\wedge$, which does preserve quasi-isomorphisms and which is functorial for arbitrary $\Coo$-morphisms, cf.~\cref{sect:background}. It follows that if $\g$ and $\h$ are quasi-isomorphic then the completions of the bar-cobar resolutions of $\g$ and $\h$ are quasi-isomorphic, which means precisely that $\hc \g \simeq \hc \h$. In \cref{sect:thmb}, we will explain the proof of Theorem \hyperref[thm B]{B}, and in \cref{sect:homotopy complete}, we will prove \cref{homotopy-complete}, giving criteria for when a dg Lie algebra is homotopy complete. Finally, in \cref{sect:background}, we will briefly recall some background on $\infty$-coalgebras. 
\end{para}

\begin{para}
	The reader who wants to get the gist of the proofs of Theorems \hyperref[thmA intro]{A} and \hyperref[thm B]{B} with a minimum of fuss about operadic preliminaries is invited to read only the statements of \cref{retractionthm} and \cref{pbwcor}, and then proceed to \cref{sect:thma,sect:thmb}.
\end{para}

\subsec{Notation and conventions}

\begin{para}
	We always work over a field $\K$ of characteristic $0$ and in the category of chain complexes. In other words, we use homological conventions and differentials have degree $-1$. We use conventions such that the dual of a chain complex is again a chain complex. The Harrison and Hochschild cochain complexes will play a supporting role in the paper; when they are mentioned they will be considered as chain complexes via the usual convention that $C^n = C_{-n}$, and so on. All algebras and coalgebras are in chain complexes unless explicitly specified otherwise, and we often omit the adjective dg, writing e.g.\ associative algebras when speaking of associative differential graded algebras. We implicitly identify invariants and coinvariants whenever necessary.
\end{para}

\begin{para}
	We consistently apply the Koszul sign rule: the category of chain complexes is symmetric monoidal with $V\otimes W\cong W\otimes V$ given by sending $v\otimes w$ to $(-1)^{\vert v\vert \vert w\vert }w\otimes v$. We denote by $s$ a formal element of degree $1$ and write $sV\coloneqq \K s\otimes V$ for the suspension of a chain complex $V$. The dual of $s$ is denoted by $s^{-1}$, so that $s^{-1}s = 1 = -ss^{-1}$.
\end{para}

\begin{para}\label{conventions}
	We try to follow the notations of \cite{lodayvallette} as closely as possible when talking about operads. All cooperads are conilpotent. Unless explicitly specified otherwise, associative and commutative algebras are non-unital, and coassociative and cocommutative coalgebras are non-counital.
\end{para}

\subsec{Acknowledgements}

\begin{para}
	The present paper grew out of a series of discussions in which we tried to understand what is the ``right'' setting to understand the results of Saleh's paper \cite{saleh}. We are grateful to Bashar Saleh, who initially participated in this project, for his role in these discussions. We are also grateful to Thomas Willwacher for helping us find a gap in the original version of the paper, and to Johan Alm, Alexander Berglund, Vladimir Dotsenko, and Bruno Vallette for mathematical comments. We thank an anonymous referee for pointing us to \cite[Corollary 5.4.4.7]{HA}. 
\end{para}

\begin{para}
	Ricardo Campos  was affiliated to Universit\'e Paris 13 when the current project begun and acknowledges support by the Swiss National Science Foundation Early Postdoc.Mobility grant P2EZP2\_174718. Dan Petersen was funded by ERC-2017-STG 759082 and by a Wallenberg Academy Fellowship. Daniel Robert-Nicoud was affiliated to Universit\'e Paris 13 when the current project begun and acknowledges gratefully the support of grants from the R\'egion Ile de France and the grant ANR-14-CE25-0008-01 project SAT. Felix Wierstra was partially supported by grant number GA CR P201/12/G028 from the Czech research council and grant number VI.Veni.202.046 from the Dutch Research Organisation (NWO), he further thanks the Max Planck Institute for Mathematics for their hospitality and excellent working conditions.
\end{para}

\section{Some deformation theory} \label{section1}

\begin{para}
	A famous principle, due to Deligne and Drinfeld and developed by many others, assigns to a dg Lie algebra a ``deformation problem'', in which the solutions to the deformation problem are Maurer--Cartan elements and deformation equivalence of solutions is defined by the action of the group obtained by exponentiating the degree $0$ elements. Any deformation problem in characteristic zero arises in this way, according to an informal principle which is now a theorem of Pridham--Lurie \cite{pridham,lurieICM}. We will only require a tiny fragment of the general theory, which we recall below; for an introductory textbook account see e.g.\ \cite{manettilectures}. However, it is worth pointing out that our set-up is quite different to the one considered in the references mentioned above: instead of considering functors of Artin rings, our dg Lie algebras have complete filtrations which make the required power series converge. So strictly speaking we will never write down an actual deformation functor.
\end{para}

\begin{para}\label{complete dgla}
	Let $\g$ be a dg Lie algebra equipped with a complete Hausdorff descending filtration
	\[
	\g = \F^1\g \supseteq \F^2\g \supseteq \cdots
	\]
	such that $d(\F^p \g) \subseteq \F^p \g$ and $[\F^p\g,\F^q\g] \subseteq \F^{p+q}\g$. The set of degree $0$ elements $\g_0$ can be made into a group, called the \emph{gauge group} of $\g$, using the Baker--Campbell--Hausdorff formula
	\[
	\mathrm{BCH}(a,b) = a+b+\frac{1}{2}[a,b] + \cdots
	\]
	in which the higher order terms are given by higher order nested brackets of $a$ and $b$. Given $a \in \g_0$ we write $\exp(a)$ for the corresponding group element. The series converges, since $\g = \F^1\g$ and the filtration is complete. The only fact about the Baker--Campbell--Hausdorff formula we will need in this article is that if $a \in\F^n\g_0$ and $b \in \g_0$ then
	\[
	\mathrm{BCH}(a,b) \equiv a+b \pmod {\F^{n+1}\g}.
	\]
\end{para}

\begin{para}\label{MC-def}
	Let $\MC(\g)$ be the set of solutions to the \emph{Maurer--Cartan equation} 
	\[
	dx + \frac{1}{2}[x,x] = 0
	\]
	in $\g_{-1}$. If $x \in \MC(\g)$, we may define a ``twisted'' differential $d_x$ on $\g$ by $d_x = [x,-]+d$. Then $(\g,[-,-],d_x)$ is again a complete filtered dg Lie algebra with the same underlying filtration, which we denote by $\g^x$. Then $y \in \g$ is a Maurer--Cartan element if and only if $y-x$ is a Maurer--Cartan element in $\g^x$.
\end{para}

\begin{para}
	The gauge group acts on $\MC(\g)$ by
	\[
	\exp(a)\cdot x = x - \sum_{n \geq 0}\frac{([a,-])^n}{(n+1)!} d_x(a).
	\]
	Two Maurer--Cartan elements are said to be \emph{gauge equivalent} if they differ by an element of $\g_0$ in this way. The only fact about the gauge action we will need here is that if $da$ and $x$ are in $\F^n\g_{-1}$ then
	\[
	\exp(a) \cdot x \equiv x - da \pmod {\F^{n+1}\g}.
	\]
\end{para}

\begin{para}
	The main result of the present section is the following one.
\end{para}

\begin{prop}\label{thm:gauges to 0}
	Let $\h \subseteq \g$ be a Lie subalgebra. Suppose that $\h$ is a retract of $\g$ as a filtered complex, meaning that there is a filtration-preserving chain map $s \colon \g \to \h$ whose restriction to $\h$ is the identity. Let $x \in \mathrm{MC}(\h)$, and suppose there is a gauge equivalence between $x$ and $0$ given by an element $a\in \g_0$.  Then $x$ is also gauge equivalent to $0$ in $\h$.
\end{prop}

\begin{proof}
	We write $x_1\coloneqq x$ and $a_1\coloneqq a$, and we define inductively the following sequence of elements for $n \geq 1$:
	\[
	a_{n+1} = \mathrm{BCH}(a_n,-s(a_n)), \qquad \text{and} \qquad x_{n+1} = \exp(s(a_n)) \cdot x_n.
	\]
	By construction, the element $x_n$ is gauge equivalent to $x_{n+1}$ via the gauge $s(a_n)$ for all $n$, which lives in $\h$. Each $x_n$ is also gauge equivalent to $0$ via the gauge $a_n$, which instead is in general only an element of $\g$. 
	
	\medskip
	
	We claim that $s(a_n)$, $da_n$, and $x_n$ are in $\F^n\g$ for all $n$. In particular, all three sequences converge to zero. We prove this by induction on $n$, the base case $n=1$ being clear.
	
	\medskip
	
	For the first claim, suppose that $s(a_n) \in\F^n\g$. Then we have 
	\[
	a_{n+1} \equiv  a_n - s(a_n) \pmod {\F^{n+1}\g}.
	\]
	It follows that
	\[
	s(a_{n+1}) \equiv s(a_n)-s(a_n) \equiv 0 \pmod {\F^{n+1}\g}.
	\]
	Here we used the fact that $s(s(x))=s(x)$ for all $x \in \g$.
	
	\medskip
	
	The second and third claims are proven in tandem. Suppose that $x_n$ and $da_n$ are in $\F^n\g$. Consider the equation $\exp(a_n) \cdot x_n = 0$ modulo $\F^{n+1}\g$ to get
	\[
	x_n \equiv da_n \pmod {\F^{n+1}\g}.
	\]
	Since $x_n \in \h$ we have $s(x_n)=x_n$. It follows that $s(da_n)=ds(a_n)$ is also equivalent to $x_n$, modulo $\F^{n+1}\g$. Thus
	\[
	x_{n+1} \equiv x_n - ds(a_n) \equiv 0 \pmod {\F^{n+1}\g}.
	\]
	Moreover, we have the identity
	\[
	d a_{n+1} \equiv da_n - ds(a_n) \pmod {F^{n+1}\g}.
	\]
	But we just saw from the equation
	\[
	x_n \equiv da_n \pmod {\F^{n+1}\g}
	\]
	that
	\[
	da_n \equiv ds(a_n) \pmod {\F^{n+1}\g},
	\]
	so that $da_{n+1} \in \F^{n+1}\g$, as claimed.
	
	\medskip
	
	It follows that $x_1$ is gauge trivial in $\h$. Indeed, all elements of the sequence $x_1,x_2,x_3, \ldots$ in $\h$ are gauge equivalent to each other in $\h$ by construction, since the gauge taking $x_n$ to $x_{n+1}$ is given by an element of $\h$. Since the sequence of gauges converges to the identity in the group, we may consider the (ordered) product $\prod_{n=1}^\infty \exp(s(a_n))$, which is now a well defined gauge from $x_1$ to $0$. 
\end{proof}

\begin{thm}\label{retractionthm}
	Let $\h \subseteq \g$ be a dg Lie subalgebra. Suppose that $\h$ is a retract of $\g$ as a filtered $\h$-module, that is, there is a filtration-preserving chain map $s \colon \g \to \h$ whose restriction to $\h$ is the identity map  and such that $s([x,y]) = [s(x),y]$ for all $x \in \g$ and $y \in \h$.  If $x$ and $y$ are Maurer--Cartan elements of $\h$ which are gauge equivalent in $\g$, then they are gauge equivalent in $\h$. 
\end{thm}

\begin{proof}
	This result reduces to \cref{thm:gauges to 0} by replacing the differentials in $\h$ and $\g$ with the twisted differential $d_y$ (\S \ref{MC-def}). The fact that $s$ is an $\h$-module morphism implies in particular that $s$ is a chain map with respect to the twisted differentials.
\end{proof}

\begin{rem}
	To any complete Lie algebra $\g$ one can associate a Kan complex of Maurer--Cartan elements $\mathrm{MC}_\bullet(\g)$ \cite{hinich97}, which contains all the information about the deformation theory encoded by $\g$. In particular,  $\pi_0(\mathrm{MC}_\bullet(\g))$ is in bijection with the set of Maurer--Cartan elements of $\g$ modulo gauge equivalence. Thus \cref{retractionthm} states that if $\h\to\g$ is a morphism of Lie algebras which is split injective as a map of $\h$-modules, then the induced map
	\[
	\pi_0(\mathrm{MC}_\bullet(\h))\longrightarrow\pi_0(\mathrm{MC}_\bullet(\g))
	\]
	is injective. It is natural to ask what happens for the higher homotopy groups. This question is answered by a theorem of Berglund \cite[Thm. 5.5]{berglund2015rational}, which gives an identification
	\[
	\pi_n(\mathrm{MC}_\bullet(\g),x)\cong H_{n-1}(\g^x),\quad n\ge1,
	\]
	functorial in $\g$, for any basepoint $x\in\mathrm{MC}(\g)$. It follows that under the hypotheses of \cref{retractionthm} we have injections
	\[
	\pi_n(\mathrm{MC}_\bullet(\h))\longrightarrow\pi_n(\mathrm{MC}_\bullet(\g))
	\]
	for any basepoint $x \in \mathrm{MC}(\h)$ and any $n\ge0$, since the assumptions imply that $\h^x$ is a direct summand of $\g^x$ as a chain complex. 
\end{rem}

\section{A consequence of the PBW theorem}\label{section2}

\begin{para}The goal of this section is to prove \cref{pbwiso}, giving a simple direct sum decomposition of the associative operad, considered as an infinitesimal bimodule over the Lie operad. This result may be considered as a refined form of Gerstenhaber--Schack's Hodge decomposition of the Hochschild complex \cite{gerstenhaberschack}. \cref{pbwiso} is not new --- it seems to have first been proven by Griffin \cite{griffincomodules}, who showed it by explicitly verifying that the Eulerian idempotents used by Gerstenhaber--Schack define morphisms of infinitesimal bimodules. Griffin's paper also explains in detail the relationship with the Hodge decomposition of the Hochschild complex. We will give a short and self-contained proof, explaining that the result can also been seen as a consequence of the Poincar\'e--Birkhoff--Witt theorem. 
	\end{para}

\begin{para}
	Recall that an {operad} can be defined as a monoid in a certain monoidal category: the category of $\mathbb S$-modules, with monoidal structure given by the \emph{composite product} $\circ$ \cite[Section 5.2]{lodayvallette}. As such there are evident notions of \emph{left} and \emph{right modules} over an operad $\P$: an $\mathbb S$-module $\M$ is a left (resp.\ right) $\P$-module if it is equipped with maps $\P \circ \M \to \M$ (resp. $\M \circ \P \to \M$) satisfying axioms of associativity and unit. If $\M$ has commuting structures of a left $\P$-module and a right $\Q$-module we say that it is a $(\P,\Q)$-bimodule.
\end{para}

\begin{para}
	The category of $\mathbb S$-modules is symmetric monoidal with respect to the tensor product (Day convolution) of $\mathbb S$-modules. If $\M$ and $\N$ are right $\Q$-modules, then $\M \otimes \N$ is again a right $\Q$-module in a natural way, making the category of right $\Q$-modules itself symmetric monoidal. The category of $(\P,\Q)$-bimodules is equivalent to the category of $\P$-algebras in the symmetric monoidal category of right $\Q$-modules \cite[Chapter 9]{fressemodules}. 
\end{para}

\begin{para}
	One can also define the \emph{infinitesimal composite product} $\circ_{(1)}$ of two $\mathbb S$-modules \cite[Section 6.1.1]{lodayvallette}. If $\P$ is an operad, an \emph{infinitesimal left} (resp. \emph{right}) \emph{module} is an $\mathbb S$-module $\M$ equipped with a map $\P \circ_{(1)} \M \to \M$ (resp. $\M \circ_{(1)} \P \to \M$) satisfying the analogous unit and associativity axioms. The notion of infinitesimal right module is equivalent to the usual notion of right module, but for left modules the two are strictly different. Moreover, neither notion is stronger or weaker than the other.
\end{para}

\begin{para}\label{inf bimod discussion}
	Let $f\colon \P \to \Q$ be a morphism of operads. Then $\Q$ becomes both a $\P$-bimodule and an infinitesimal $\P$-bimodule. When we consider $\Q$ as a left $\P$-module, we are considering morphisms
	\[
	\P(k) \otimes \Big( \Q(n_1) \otimes \cdots \otimes \Q(n_k) \Big) \longrightarrow \Q(n_1+\cdots+n_k),
	\]
	and when we consider $\Q$ as an infinitesimal left $\P$-module we are considering instead the morphisms
	\[
	\P(k) \otimes \Big(\P(n_1) \otimes \cdots \otimes \Q(n_i)\otimes \cdots \otimes \P(n_k)\Big) \longrightarrow \Q(n_1+\cdots+n_k).
	\]
	This means that considering $\Q$ as a left $\P$-module is equivalent to considering $\Q$ as an algebra over the operad $\P$ in the category of $\mathbb S$-modules, and considering $\Q$ as an infinitesimal left $\P$-module is equivalent to considering $\Q$ as a module over $\P$, where $\P$ is considered as an algebra over itself in the category of $\mathbb S$-modules.
\end{para}

\begin{para}
	There is a pushforward functor $f_!$ from $\P$-algebras to $\Q$-algebras which is left adjoint to the functor $f^*$ restricting a $\Q$-algebra structure to a $\P$-algebra structure along $f$. If $A$ is a $\P$-algebra, then $f_!A$ is the $\Q$-algebra defined as the coequalizer of the two natural arrows
	\[
	\Q(\P(A))\,\, \substack{\longrightarrow\\[-1em] \longrightarrow }\,\, \Q(A)
	\]
	given by applying the $\P$-algebra structure of $A$, and by mapping $\P$ to $\Q$ using $f$ and then applying the operadic composition in $\Q$, respectively. This coequalizer can also be written as a ``relative composite product'' $\Q \circ_\P A$. If we consider the operad $\P$ itself as a $\P$-algebra in right $\P$-modules, then $f_!\P$ is the $\Q$-algebra in right $\P$-modules given by $\Q$ itself, considered as a $(\Q,\P)$-bimodule.
\end{para}

\begin{para}\label{morphismfromlietoass}
	An important example of this pushforward functor is given by the universal enveloping algebra. Any unital associative algebra may be considered as a Lie algebra, with bracket given by the commutator; this forgetful functor corresponds to a morphism of operads $\lie \to \ass^+$, where $\lie$ is the Lie operad and $\ass^+$ is the operad of unital associative algebras. The pushforward gives a functor from Lie algebras to unital associative algebras, which is precisely the usual universal enveloping algebra construction.
\end{para}

\begin{para}What will be more important for us in this paper is the operad $\ass$ of non-unital associative algebras. The pushforward along $\lie \to \ass$ maps a Lie algebra to the augmentation ideal of its universal enveloping algebra, and the pushforward along $\ass \to \ass^+$ is the functor which freely adjoins a unit to a non-unital algebra.
\end{para}

\begin{para}
	For a Lie algebra $\mathfrak g$, the category of $\mathfrak g$-modules is symmetric monoidal. In particular, if $M$ is a $\mathfrak g$-module, then there is a natural $\mathfrak g$-module structure on the symmetric algebra $\operatorname{Sym}(M)$. When $\g=\lie$ is the Lie operad, considered as an algebra over itself in $\mathbb S$-modules, this says that if $\Q$ is an infinitesimal $\lie$-bimodule, then $\operatorname{Sym}(\Q)$ is again naturally an infinitesimal $\lie$-bimodule, where $\operatorname{Sym}(\Q)=\bigoplus_{k\geq 0} \sym^k(\Q)$ denotes the direct sum of all symmetric powers of $\Q$, taken in the category of $\mathbb S$-modules. This observation, and those of \S\ref{inf bimod discussion}, make the following proposition meaningful:
\end{para}

\begin{prop}\label{pbwiso}
	There is an isomorphism of infinitesimal $\lie$-bimodules
	$
	\ass^+ \cong \operatorname{Sym}(\lie)$. Similarly, $\ass \cong \bigoplus_{k \geq 1} \sym^k(\lie)$ as infinitesimal $\lie$-bimodules. 
\end{prop}

\begin{proof}
	Consider $\lie$ as a bimodule over itself. Then the $(\ass^+,\lie)$-bimodule given by $f_!\lie$, i.e.\ the universal enveloping algebra of $\lie$, is given by $\ass^+$. The Poincar\'e--Birkhoff--Witt theorem states that for any Lie algebra $\g$ in characteristic zero there is an isomorphism of $\g$-modules
	\[
	U\g\cong\sym(\g).
	\]
	This theorem is true for Lie algebras in any $\K$-linear symmetric monoidal abelian category \cite[\S 1.3.7]{delignemorgan}. In particular, $\ass^+$ is isomorphic to the symmetric algebra on $\lie$, considered as a module over the Lie algebra $\lie$ in the category of right $\lie$-modules. But a module over the Lie algebra $\lie$ in the category of right $\lie$-modules is exactly the same thing as an infinitesimal $\lie$-bimodule.
\end{proof}

\begin{rem}
	If we disregard the bimodule structure, \cref{pbwiso} expresses the well-known fact that $\ass^+\cong\com^+\circ\lie$. We write $\sym(\lie)$ rather than $\com^+\circ\lie$ because the latter notation obscures the infinitesimal bimodule structure. %Indeed, if $M$ is a module over a Lie algebra $\g$, then there is also a natural $\g$-module structure on $\sym(M)$. We are applying this fact to the Lie algebra $\lie$ in the category of right $\lie$ modules, considered as a module over itself.
\end{rem}

\begin{cor}\label{pbwcor}
	Let $f \colon \lie \to \ass$ be the natural morphism described in \S \ref{morphismfromlietoass}. There is a morphism of infinitesimal $\lie$-bimodules $s \colon \ass \to \lie$ such that $s \circ f = \operatorname{id}_{\lie}$.
\end{cor}

\begin{proof}
	Indeed, $s$ is given by projecting onto the summand $\sym^1(\lie)=\lie$. 
\end{proof}

\begin{para}
	In the next section, we will consider the \emph{deformation complexes} of $\Aoo$-deformations and $\Coo$-defor\-ma\-tions of a $\Coo$-algebra. These are (essentially) the \emph{Hochschild cochain complex} and the \emph{Harrison cochain complex}, respectively. The isomorphism of infinitesimal $\lie$-bimodules
	\[
	\ass \cong \bigoplus_{k \geq 1} \sym^k(\lie)
	\]
	gives rise to a direct sum decomposition of the Hochschild cochains of a commutative or $\Coo$-algebra, for which the $k=1$ summand $\sym^1(\lie)=\lie$ is identified with the Harrison cochains. This decomposition coincides with the Hodge decomposition of Hochschild cohomology of Quillen \cite[\S 8]{quillencommutativerings} and Gerstenhaber--Schack \cite{gerstenhaberschack}. The relationship between the Hodge decomposition and the Poincar\'e--Birkhoff--Witt theorem seems to have first been made explicit by Bergeron and Wolfgang \cite{bergeronwolfgang}, although in a different form than the one found here. The only fact we will need for the proofs of Theorems \hyperref[thmA main]{A} and \hyperref[thm B stated again]{B} is Corollary \ref{pbwcor}, which in this context says that the Hochschild cochains retracts onto the Harrison cochains, and in particular that Harrison cohomology is a direct summand of Hochschild cohomology \cite{barrhochschild}. However, it does not seem possible to deduce Theorems \hyperref[thmA main]{A} and \hyperref[thm B stated again]{B} purely from the fact that Harrison cohomology injects into Hochschild cohomology; we really do need the stronger statement that there exists a splitting on the chain level (or a splitting of infinitesimal operadic bimodules). By contrast, Saleh \cite{saleh} proves the weaker statement that if a $\Coo$-algebra is formal as an $\Aoo$-algebra then it is also formal as a $\Coo$-algebra, using only the fact that Harrison cohomology is a direct summand of Hochschild cohomology.
\end{para}

\begin{para}
	Proposition \ref{pbwiso}, stated only for right modules, can be found in \cite[Lemma 10.2.6]{fressemodules}. Dotsenko and Tamaroff \cite{dotsenkotamaroff} explain more generally that a morphism of operads $\P \to \Q$ satisfies a PBW-type theorem if and only if $\Q$ is free as a right $\P$-module. Fresse and Dotsenko--Tamaroff consider right modules instead of infinitesimal bimodules, since for them the statement of the PBW theorem is that $U\g \cong \sym(\g)$ as vector spaces, not as $\g$-modules.	
\end{para}

\section{Proof of Theorem A}\label{sect:thma}

\begin{para}
The goal of this section is to prove Theorem \hyperref[thmA main]{A}: 
\end{para}

\begin{thmA}\label{thmA main}
	Two commutative dg algebras $A$ and $B$ are quasi-isomorphic if and only if they are quasi-isomorphic as associative dg algebras.
\end{thmA}

\begin{para}\label{lurie1}
	We remind the reader of our convention (\S\ref{conventions}) that all dg algebras are assumed to be non-unital unless specified otherwise. It is natural to ask whether Theorem \hyperref[thmA main]{A} hold also in the unital setting, i.e.\ whether two unital commutative dg algebras are quasi-isomorphic whenever they are quasi-isomorphic as unital associative dg algebras. The answer is \emph{yes} --- in fact, this is implied by Theorem \hyperref[thmA main]{A}, together with a general result of Lurie \cite[Corollary 5.4.4.7]{HA}, as was pointed out to us by an anonymous referee. Informally, Lurie's result states that in order to give a unital structure to a non-unital algebra in an $\infty$-category it is enough to exhibit a \emph{quasi-unit}, which is roughly a unit in the homotopy category, and that a unital structure is unique if it exists.
	% a non-unital commutative algebra in an $\infty$-category can be uniquely endowed with a unit, if possible, and that to give a unital structure it suffices to construct a \emph{quasi-unit}, which is roughly a unit in the homotopy category.
	A precise statement (a weakening of \cite[Corollary 5.4.4.7]{HA}) is that if $\mathcal C$ is a symmetric monoidal $\infty$-category, $\mathrm{CAlg}(\mathcal C)$ denotes the $\infty$-category of $\mathsf E_\infty$-algebras in $\mathcal C$, and $\mathrm{CAlg}^{\mathrm{nu}}(\mathcal C)$ denotes the $\infty$-category of non-unital $\mathsf E_\infty$-algebras in $\mathcal C$, then the forgetful functor
	\[
	\mathrm{CAlg}(\mathcal C)^\simeq \longrightarrow \mathrm{CAlg}^{\mathrm{nu}}(\mathcal C)^\simeq
	\]
	can be identified with the inclusion of the full subcategory spanned by those non-unital $\mathsf E_\infty$-algebras which admit a quasi-unit. Here $(-)^\simeq$ denotes the maximal Kan subcomplex of an $\infty$-category, i.e.\ the subcategory with the same objects and only isomorphisms. In particular, the functor is injective on isomorphism classes.
\end{para}
	
\begin{para}\label{lurie2}
	Specializing to the setting where $\mathcal C$ is the $\infty$-category of unbounded $\K$-chain complexes modulo quasi-isomorphism, Lurie's result shows in particular that two unital commutative dg algebras $A$ and $B$ are quasi-isomorphic if and only if they are quasi-isomorphic as non-unital commutative dg algebras. (Indeed, the homotopy theories of $\mathsf E_\infty$-algebras and commutative algebras are equivalent over a field of characteristic zero.) With this fact in place it is clear that the unital version of Theorem \hyperref[thmA main]{A} follows from the non-unital version, whose proof will take up the rest of this section.
\end{para}

\begin{para}\label{properties}
	Instead of working with commutative  algebras we work in the larger category of $\Coo$-algebras and $\Coo$-morphisms, also known as $\infty$-morphisms of $\Coo$-algebras. We denote $\Coo$-morphisms by a squiggly arrow. This category has the following useful properties:
	\begin{enumerate}
		\item \cite[Theorem 10.3.10]{lodayvallette} Any $\Coo$-algebra is $\Coo$-quasi-isomorphic to a \emph{minimal} $\Coo$-algebra, i.e.\ a $\Coo$-algebra with zero differential.
		\item \cite[Theorem 10.4.4]{lodayvallette} If two $\Coo$-algebras $A$ and $B$ are quasi-isomorphic, then there exists a $\Coo$-quasi-isomorphism\footnote{As opposed to a zig-zag of quasi-isomorphisms.} $A \rightsquigarrow B$.
		\item  \cite[Theorem 11.4.8]{lodayvallette} Two commutative dg algebras are quasi-isomorphic if and only if they are $\Coo$-quasi-isomorphic.
	\end{enumerate}
\end{para}

\begin{para}
	We will similarly work with $\Aoo$-algebras instead of associative algebras; they satisfy evident analogues properties (1'), (2') and (3').
\end{para}

\begin{para}\label{firstreduction p1}
	Suppose that we are given two commutative dg algebras $A$ and $B$ that are quasi-isomorphic as \emph{associative} dg algebras. Our goal is to show that they are quasi-isomorphic as commutative dg algebras as well. By (1), we may assume that $A$ and $B$ are minimal. By (2'), there is an $\Aoo$-quasi-isomorphism $A\rightsquigarrow B$, and by (3), the proof of Theorem \hyperref[thmA main]{A} is reduced to showing the existence of a $\Coo$-quasi-isomorphism $A\rightsquigarrow B$.
\end{para}

\begin{para}\label{firstreduction p2}
	We can make the following further simplification. Note that a quasi-isomorphism between chain complexes with vanishing differential is just an isomorphism. Therefore by the minimality assumption, the underlying graded vector spaces of $A$ and $B$ are isomorphic, an isomorphism being given by the first component of the given $\Aoo$-morphism. By ``transport of structure'' along this isomorphism we thereby reduce to the case where $A$ and $B$ are minimal $\Coo$-algebras with the same underlying graded vector space that are linked by an $\Aoo$-morphism whose linear component is given by the identity, i.e.\ what is called an \emph{$\Aoo$-isotopy}.
\end{para}

\begin{para}
	Putting all of this together, we see that Theorem \hyperref[thmA main]{A} is implied by the following statement.
\end{para}

\begin{prop}\label{prop:reduced thm A}
	Let $V$ be a chain complex. Suppose that we are given two $\Coo$-algebra structures on $V$, and an $\Aoo$-isotopy between them. Then there also exists a $\Coo$-isotopy between them. 
\end{prop}

\begin{para}
	For the proof of \cref{prop:reduced thm A}, we will apply \cref{pbwcor} to the \emph{deformation complexes} $\Def_{\Aoo}(V)$ and $\Def_{\Coo}(V)$ of $\Aoo$-algebra and $\Coo$-algebra structures on $V$. These are complete filtered graded dg Lie algebras whose Maurer--Cartan elements are the $\Aoo$-algebra (resp. $\Coo$-algebra) structures on $V$, and whose gauge equivalences are $\Aoo$-isotopies (resp. $\Coo$-isotopies). Elements of the deformation complexes are given by collections of equivariant maps: 
	\[
	\Def_{\Aoo}(V)\coloneqq\prod_{n \geq 2}\hom_{\mathbb{S}_n}\left(\susp\coass(n),\hom_\K(V^{\otimes n}, V)\right)
	\]
	and
	\[
	\Def_{\Coo}(V)\coloneqq\prod_{n \geq 2}\hom_{\mathbb{S}_n}\left(\susp\colie(n),\hom_\K(V^{\otimes n}, V)\right).
	\]
	These complexes are filtered by
	\[
	\F^p\Def_{\Aoo}(V)\coloneqq\prod_{n\ge p+1}\hom_{\mathbb{S}_n}\left(\susp\coass(n),\hom_\K(V^{\otimes n}, V)\right),
	\]
	and similarly for $\Def_{\Coo}(V)$.
\end{para}

\begin{para}
	Here, $\susp\coass$ is the Koszul dual cooperad of $\ass$, given by the operadic suspension \cite[Section 7.2.2]{lodayvallette} of the cooperad $\coass$ encoding coassociative coalgebras. Similarly, $\susp\colie$ is the Koszul dual of $\com$, given by the suspension of the cooperad $\colie$ encoding Lie coalgebras.
\end{para}

\begin{para}\label{description of convolution dgla}
	To describe the Lie algebra structure on the deformation complexes, and to see that $\Def_{\Coo}(V)$ is a Lie subalgebra of $\Def_{\Aoo}(V)$, it is useful to put ourselves in a more general situation. If $\C$ is a dg cooperad and $\P$ is a dg operad, then we can define a complete filtered dg Lie algebra
	\[
	\overline\hom_{\mathbb{S}}(\C,\P) = \prod_{n \geq 2} \hom_{\mathbb{S}_n} (\C(n),\P(n))
	\]
	which is called the \emph{convolution Lie algebra} of $\C$ and $\P$. This construction is covariantly functorial in $\P$ and contravariantly functorial in $\C$. There is a binary operation $\star$ on $\overline\hom_{\mathbb{S}}(\C,\P)$ which can be heuristically described as follows: if $f, g \in \overline\hom_{\mathbb{S}}(\C,\P)$, then $f \star g$ is the composition
	\[
	\C \longrightarrow \C \circ_{(1)} \C \xrightarrow{f \circ_{(1)} g} \P \circ_{(1)} \P \longrightarrow \P,
	\]
	where the first and last arrow are given by the infinitesimal cocomposition (resp.\ composition) of $\C$ (resp.\ $\P$). See  \cite[Section 6.4]{lodayvallette} for a precise description. The Lie bracket is then defined by $[f,g] = f \star g - (-1)^{\vert g \vert \vert f \vert }g \star f$. The deformation complexes can now be defined as $\Def_{\Coo}(V) = \overline\hom_{\mathbb{S}}(\susp\colie,\End_V)$ and $\Def_{\Aoo}(V) = \overline\hom_{\mathbb{S}}(\susp\coass,\End_V)$, where $\End_V$ is the endomorphism operad of $V$. Dualizing the natural injection $\lie\to\ass$ defines a surjection $\coass \to \colie$ which induces the embedding of $\Def_{\Coo}(V)$ into $\Def_{\Aoo}(V)$.
\end{para}

\begin{thm} \cite[Theorem 3]{preliedeformationtheory}
	The Maurer--Cartan elements of $\Def_{\Aoo}(V)$ are in bijection with the set of $\Aoo$-algebra structures on $V$, and the group of gauge equivalences coincides with the group of $\Aoo$-isotopies. Similarly the Maurer--Cartan set of $\Def_{\Coo}(V)$ is the set of $\Coo$-algebra structures on $V$, and the group of gauge equivalences is the group of $\Coo$-isotopies. 
\end{thm}

\begin{para}
	The first part of the theorem is well known: a Maurer--Cartan element of $\overline\hom_{\mathbb{S}}(\C,\P)$ is by definition an \emph{operadic twisting morphism} from $\C$ to $\P$, and if $\P$ is a Koszul operad and $V$ is a chain complex, then a twisting morphism from the Koszul dual cooperad $\P^\antishriek$ to $\End_V$ is the same thing as a $\P_\infty$-algebra structure on $V$ \cite[Section 10.1]{lodayvallette}. For the second half of the theorem there is an obvious bijection between the gauge group and the set of $\infty$-isotopies; the nontrivial part of the theorem is to show that the group operations and the group action on the Maurer--Cartan set (which on one side are defined by sums over trees formulas, and on the other side are defined by the Baker--Campbell--Hausdorff formula) are identified under the obvious bijection.
\end{para}

\begin{rem}
	An $\Aoo$-structure on $V$ corresponds to a Maurer--Cartan element in $\Def_{\Aoo}(V)$, and twisting by this Maurer--Cartan element (\S\ref{MC-def}) defines a differential on $\Def_{\Aoo}(V)$. Up to a degree shift and the fact that the $n=1$ component $\hom_\K(V,V)$ is missing, $\Def_{\Aoo}(V)$ with this differential is the Hochschild cochain complex of the $\Aoo$-algebra. Similarly, if we twist $\Def_{\Coo}(V)$ by the Maurer--Cartan element corresponding to a $\Coo$-algebra structure on $V$, we recover the Harrison cochain complex of the $\Coo$-algebra.
\end{rem}

\begin{proof}[Proof of \cref{prop:reduced thm A}]
	We can rephrase the statement in terms of deformation complexes as follows. We are given two Maurer--Cartan elements in the Lie algebra $\Def_{\Coo}(V)$ which are gauge equivalent in the bigger Lie algebra $\Def_{\Aoo}(V)$. We need to prove that the two Maurer--Cartan elements are also gauge equivalent in $\Def_{\Coo}(V)$. This puts us in the situation considered in \cref{section1}, and by \cref{retractionthm} we are done if we can prove that $\Def_{\Aoo}(V)$ retracts onto $\Def_{\Coo}(V)$ as a filtered $\Def_{\Coo}(V)$-module.
	
	\medskip
	
	As already mentioned above, the inclusion $\Def_{\Coo}(V) \hookrightarrow \Def_{\Aoo}(V)$ is induced by the dual of the map $\lie \to \ass$. Clearly any retraction of $\mathbb{S}$-modules $s\colon\ass\to\lie$ will induce a retraction of filtered complexes from $\Def_{\Aoo}(V)$ to $\Def_{\Coo}(V)$, but a priori we will not have any compatibility with the Lie brackets. We claim that if $s$ is a morphism of infinitesimal $\lie$-bimodules, then the induced map
	\[
	\Def_{\Aoo}(V)\longrightarrow\Def_{\Coo}(V)
	\]
	is a morphism of $\Def_{\Coo}(V)$-modules. Showing this will complete the proof of Theorem \hyperref[thmA main]{A} since by \cref{pbwcor} we have such a morphism $s\colon\ass\to\lie$ of infinitesimal $\lie$-bimodules.

	\medskip
	Again it is useful to put ourselves in a slightly more general setting. If $\C$ is a cooperad and $\P$ is an operad, then we have the convolution Lie algebra $\overline{\hom}_{\mathbb S}(\C,\P)$; if $\M$ is an infinitesimal $\C$-bicomodule and $\N$ is an infinitesimal $\P$-bimodule then 
	\[
	\overline{\hom}_{\mathbb S}(\M,\N) = \prod_{n \geq 2}\hom_{\mathbb S_n}(\M(n),\N(n))
	\]
	is naturally a filtered module over the Lie algebra $\overline{\hom}_{\mathbb S}(\C,\P)$. Specifically, if $f \in \overline{\hom}_{\mathbb S}(\C,\P)$ and $\xi \in \overline{\hom}_{\mathbb S}(\M,\N)$ then we set
	\[f\cdot \xi = f \star \xi - (-1)^{|f||\xi|}\xi \star f, \]
	where $f \star \xi$ denotes the composition
	\[
	\M \longrightarrow \C \circ_{(1)} \M \xrightarrow{f \circ_{(1)} \xi} \N \circ_{(1)} \P \longrightarrow \N,
	\]
	with the first and last arrow being given by the infinitesimal left (co)module structures of $\M$ and $\N$, respectively. One defines $\xi \star f$ similarly, using instead the right (co)module structure of $\M$ and $\N$. This construction is functorial in $\M$ and $\N$. Note in particular that a morphism of cooperads $\D \to \C$ makes $\D$ into an infinitesimal bicomodule over $\C$, which means that $\overline{\hom}_{\mathbb S}(\D,\P)$ is both a Lie algebra equipped with a morphism from $\overline{\hom}_{\mathbb S}(\C,\P)$, as well as a module over the  Lie algebra $\overline{\hom}_{\mathbb S}(\C,\P)$. These two structures are compatible with each other, in the sense that the module structure on $\overline{\hom}_{\mathbb S}(\D,\P)$ deduced from the infinitesimal bicomodule structure on $\D$ agrees with the one obtained from the pullback morphism $\overline{\hom}_{\mathbb S}(\C,\P) \to \overline{\hom}_{\mathbb S}(\D,\P)$.
	
	\medskip
	
	The map $\lie \to \ass$ makes $\ass$ into an infinitesimal bimodule over $\lie$. Dualizing, $\coass$ becomes an infinitesimal bicomodule over $\colie$, and this defines the $\Def_{\Coo}(V)$-module structure on $\Def_{\Aoo}(V)$. Given a morphism of infinitesimal bimodules $\ass \to \lie$, we obtain by dualizing a morphism of infinitesimal bicomodules $\colie \to \coass$ and hence a morphism of $\Def_{\Coo}(V)$-modules from $\Def_{\Aoo}(V)$ to $\Def_{\Coo}(V)$. This concludes the proof of Theorem \hyperref[thmA main]{A} in the non-unital case.
\end{proof}

\begin{rem}
	In the first preprint version of this paper, we proved the unital and non-unital versions of Theorem \hyperref[thmA main]{A} by separate arguments, instead of arguing as in \S\S\ref{lurie1}--\ref{lurie2}. Indeed, it is also possible to treat the unital case by a modification of the arguments used to prove the non-unital case. To do this, one may systematically work with \emph{strictly unital} $\Aoo$- and $\Coo$-algebras. The categories of strictly unital $\Aoo$- and $\Coo$-algebras satisfy properties exactly analogous to those stated in \S\ref{properties}. Moreover, strictly unital $\infty$-algebra structures and $\infty$-isotopies are parametrized by deformation complexes exactly like those used in the non-unital case. In the same way that the deformation complexes for non-unital $\Aoo$-algebra and $\Coo$-algebra structures correspond to the Hochschild and Harrison complexes, the deformation complexes for strictly unital algebra structures correspond to the \emph{normalized} Hochschild and Harrison complexes, respectively. One difference is that the strictly unital versions of the deformation complexes are not dg Lie algebras but \emph{curved} Lie algebras, but \cref{retractionthm} is true just as well in the curved setting (cf.\ \cite[Section 5]{followup}). 
\end{rem}

\section{Proof of Theorem B}\label{sect:thmb}

\begin{para}	
	The goal of this section is to prove that two dg Lie algebras with quasi-isomorphic universal enveloping algebras have quasi-isomorphic homotopy completions (for this notion, see \cref{adic and lcs filtration}).
\end{para}

\begin{thmB}\label{thm B stated again}
	Let $\g$ and $\h$ be two dg Lie algebras. If the universal enveloping algebras $U\g$ and $U\h$ are quasi-isomorphic as unital associative dg algebras, then the homotopy completions $\hc\g$ and $\hc\h$ are quasi-isomorphic as dg Lie algebras. 
\end{thmB}

\begin{para}For the proof we will need to juggle the bar-cobar adjunction between associative algebras and coassociative coalgebras, as well as the bar-cobar adjunction between Lie algebras and cocommutative coalgebras. 	We denote these adjunctions by
	\[
	\CobarAC : \{\text{conilpotent  coassociative dg coalgebras}\}\,\, \substack{\longleftarrow\\[-1em] \longrightarrow }\,\, \{\text{associative dg algebras}\}: \BarAC,
	\]
	and
	\[
	\CobarLC : \{\text{conilpotent  cocommutative dg coalgebras}\}\,\, \substack{\longleftarrow\\[-1em] \longrightarrow }\,\, \{\text{dg Lie algebras}\}: \BarLC ,
	\]
	and we refer the reader to \cite[Chapter 11]{lodayvallette} for more details on how these functors are defined. We will in particular use the fact that the functors $\BarAC$ and $\BarLC$ preserve quasi-isomorphisms \cite[Proposition 11.2.3]{lodayvallette}, and that the counit and the unit of both adjunctions are pointwise quasi-isomorphisms \cite[Corollary 11.3.5]{lodayvallette}.	The reader should keep in mind that algebras and coalgebras are assumed to be non-unital (resp.\ non-counital) unless stated otherwise. 
\end{para}

\begin{para}
	We begin with two simple preliminary lemmas. 
\end{para}

\begin{lem}\label{lemma:U preserves qi}
	If a morphism $\g\to\h$ of dg Lie algebras is a quasi-isomorphism, then $U\g \to U\h$ is also a quasi-isomorphism.
\end{lem}

\begin{proof}
	If $\g\to\h$ is a quasi-isomorphism of Lie algebras, then $\mathrm{Sym}(\g) \to \mathrm{Sym}(\h)$ is a quasi-isomorphism of chain complexes. The statement then follows immediately from the functoriality of the Poincar\'e--Birkhoff--Witt isomorphism, see e.g.\ \cite[Thm. 2.3 of Appendix B]{quillenrationalhomotopytheory}. 
\end{proof}

\begin{para}
	It is well-known that taking the augmentation ideal gives an equivalence of categories between augmented associative algebras and non-unital associative algebras, its inverse associating to an algebra $A$ the augmented unital algebra $A^+$ obtained from $A$ by formally adding a unit. Similarly, we have an equivalence of categories between coaugmented coassociative coalgebras and non-counital coassociative coalgebras.
	
\end{para}

\begin{lem}\label{lemma:isom cobar U}
	For any cocommutative conilpotent dg coalgebra $C$ there is a natural isomorphism of augmented dg associative algebras $(\CobarAC C)^+ \cong U\CobarLC C$. 
\end{lem}

\begin{proof}
	Ignoring the cobar differentials, the result just says that the tensor algebra is canonically isomorphic to the universal enveloping algebra of the free Lie algebra. The compatibility of the isomorphism with the differentials is a computation, see e.g.\ \cite[p. 290, last paragraph]{quillenrationalhomotopytheory}.
\end{proof}

\begin{para}
	As a first step towards Theorem \hyperref[thm B stated again]{B}, we will show that if $U\g$ and $U\h$ are quasi-isomorphic as unital associative algebras, then they are also quasi-isomorphic as \emph{augmented}  associative algebras. This is a consequence of the following lemma, see also \cite[Lemma 2.1]{rileyusefi}.
\end{para}

\begin{lem}\label{different augmentations}
	Let $\g$ be a dg Lie algebra. Let $u \colon \K \to U\g$ and $\varepsilon \colon U\g \to \K$ be the unit element and augmentation of its universal enveloping algebra.  Suppose that $\overline{\varepsilon} : U\g \to \K$ is any other augmentation of  $U\g$. Then there exists an automorphism $\alpha \colon U\g \to U\g$ of unital associative algebras such that $\varepsilon = \overline{\varepsilon}\alpha$. 
\end{lem}

\begin{proof}
	Consider the composition
	\[
	\g \longrightarrow U\g \xrightarrow{\mathrm{id}- u \circ \overline{\varepsilon}}U\g,
	\]
	which is a morphism of Lie algebras. By the universal property of the enveloping algebra, this induces a morphism of unital associative  algebras $\alpha: U\g \to U\g $. We have $\varepsilon(x) = \overline{\varepsilon} \circ \alpha(x)$ for all $x \in U\g$. Indeed, since $\g$ generates $U\g$ it is enough to check this equality for $x \in \g$, in which case the identity is obvious. Moreover, $\alpha$ is an isomorphism. To see this, we start by noticing that $\alpha$ preserves the Poincar\'e--Birkhoff--Witt filtration on $U\g$, i.e.\ the filtration obtained by declaring that $F_k U\g$ is spanned by products of at most $k$ elements of $\g$. Indeed, $\alpha$ maps $\g$ into $F_1U\g$, so the result follows since $\g$ generates $U\g$. It is also straightforward to check that the induced map on the associated graded is the identity map. Since the filtration is bounded below and exhaustive, it follows that $\alpha$ is bijective.
\end{proof}

\begin{lem}\label{lemma:UL qi as unital algebras then also as augmented}
	Let $\g$ and $\h$ be dg Lie algebras. Suppose $U\g$ and $U\h$ are quasi-isomorphic as unital associative algebras. Then they are also quasi-isomorphic as augmented associative algebras. 
\end{lem}

\begin{proof}
	If $\g$ is a dg Lie algebra, we have a natural quasi-isomorphism
	\[
	\CobarLC\BarLC\g\stackrel{\sim}{\longrightarrow}\g,
	\]
	given by the counit of the bar-cobar adjunction. By \cref{lemma:U preserves qi}, this gives a quasi-isomorphism of augmented associative algebras
	\[
	U\CobarLC\BarLC\g\stackrel{\sim}{\longrightarrow}U\g .
	\]
	Therefore, it is enough to show that $U\CobarLC\BarLC\g$ and $U\CobarLC\BarLC\h$ are quasi-isomorphic as augmented associative  algebras, and then by \cref{lemma:isom cobar U} it is enough to construct such a quasi-isomorphism between  $(\CobarAC\BarLC\g)^+$ and $(\CobarAC\BarLC\h)^+$.
	
	\medskip
	
	Now we already know that $(\CobarAC\BarLC\g)^+$ and $(\CobarAC\BarLC\h)^+$ are quasi-isomorphic as \emph{unital}  algebras, since we assumed that $U\g$ and $U\h$ were quasi-isomorphic. Moreover, we may in fact assume the existence of a quasi-isomorphism of unital dg algebras
	\[
	\phi:(\CobarAC\BarLC\g)^+\stackrel{\sim}{\longrightarrow}(\CobarAC\BarLC\h)^+
	\]
	(as opposed to a zig-zag of quasi-isomorphisms). Indeed,  $(\CobarAC\BarLC\g)^+$ is a \emph{triangulated} unital associative  algebra \cite[Appendix B.6.7]{lodayvallette}, for the same reason that any bar-cobar-resolution of an algebra is triangulated. Hence one can construct construct $\phi$ by induction on the depth of the corresponding filtration of $\BarLC\g$, with the requirement that the relevant triangle commutes at the level of homology. More generally one may note that triangulated algebras are bifibrant for the model structure on unbounded unital dg algebras constructed by Hinich \cite{hinichmodelstructure}.  
	
	\medskip
	
	Now the quasi-isomorphism $\phi$ has no reason to be compatible with the two augmentations on $(\CobarAC\BarLC\g)^+$ and $(\CobarAC\BarLC\h)^+$. However, by \cref{different augmentations} we may compose $\phi$ with an automorphism of $(\CobarAC\BarLC\g)^+$ to obtain a quasi-isomorphism which preserves the augmentations, which concludes the proof.
\end{proof}

\begin{prop}\label{Cg q-iso Ch}
	Let $\g$ and $\h$ be dg Lie algebras. Suppose that $U\g$ and $U\h$ are quasi-isomorphic as unital dg associative algebras. Then $\BarLC \g$ and $\BarLC \h$ are quasi-isomorphic as conilpotent \emph{coassociative} dg coalgebras. 
\end{prop}

\begin{proof}
	By \cref{lemma:UL qi as unital algebras then also as augmented}, $U\g$ and $U\h$ are also quasi-isomorphic as augmented associative algebras. Since $\g\simeq\CobarLC\BarLC \g$ for any dg Lie algebra and the universal enveloping algebra functor preserves quasi-isomorphisms by \cref{lemma:U preserves qi}, we have
	\[
	U\CobarLC\BarLC\g\simeq U\g\simeq U\h\simeq U\CobarLC\BarLC\h.
	\]
	Then by \cref{lemma:isom cobar U} we have $(\CobarAC\BarLC\g)^+\simeq(\CobarAC\BarLC\h)^+$ as augmented associative algebras, so that we also have $\CobarAC\BarLC\g\simeq\CobarAC\BarLC\h$ as non-unital associative algebras. We now apply the bar functor $\BarAC$ to get a string of quasi-isomorphisms of {coassociative} conilpotent coalgebras
	\[
	\BarLC\g\simeq\BarAC\CobarAC\BarLC\g\simeq\BarAC\CobarAC\BarLC\h\simeq\BarLC\h
	\]
	which implies the claim.
\end{proof}

\begin{para}
	We are now in a situation entirely dual to the one considered in Theorem \hyperref[thmA main]{A}. Indeed, we have the two conilpotent cocommutative coalgebras $\BarLC\g$ and $\BarLC\h$ which are quasi-isomorphic (in fact even weakly equivalent) as conilpotent coassociative coalgebras. If we wanted to prove that $\g \simeq \h$, then a potential approach would be to try to prove a dual version of Theorem \hyperref[thmA main]{A} implying that $\BarLC\g$ and $\BarLC\h$ are already quasi-isomorphic as \emph{cocommutative} coalgebras, and then apply the cobar functor $\CobarLC$ and hope to deduce the following string of quasi-isomorphisms
	\[
	\g\simeq\CobarLC\BarLC \g\simeq\CobarLC\BarLC \h\simeq \h.
	\]
	Unfortunately there are some obstacles involved in realizing this strategy. Although we will prove an analogue of Theorem \hyperref[thmA main]{A} for coalgebras (\cref{dualThmA}), we can not show in general that $\BarLC\g$ and $\BarLC\h$ are quasi-isomorphic as cocommutative coalgebras. Moreover, even if we could,  the cobar functor does not preserve quasi-isomorphisms in general \cite[Section 2.4]{lodayvallette}. Nevertheless a version of this idea does work to prove that $\hc\g \simeq \hc\h$. 
\end{para}

\begin{para}
	The rest of this section will be devoted to deducing Theorem \hyperref[thm B stated again]{B} from \cref{Cg q-iso Ch}. This will require proving \cref{dualThmA}, a ``dual'' form of Theorem \hyperref[thmA main]{A}, which forces us to work systematically with $\Coo$- and $\Aoo$-coalgebras. The theory of $\infty$-coalgebras is less developed and less standardized than the theory of $\infty$-algebras, and one can find multiple inequivalent definitions of $\Coo$- and $\Aoo$-coalgebra being used in the literature. In \cref{sect:background} we will give a more detailed background on $\Coo$- and $\Aoo$-coalgebras. For the moment, let us just state what definition we are using and what properties of $\Coo$- and $\Aoo$-coalgebras we need, and conclude the proof of Theorem \hyperref[thm B stated again]{B}.
\end{para}

\begin{para}
	We will first need to recall the notions of completion and homotopy completion; they will be redefined in greater generality in \cref{sect:homotopy complete}.
\end{para}

\begin{defn}
	\label{adic and lcs filtration} Let $A$ be a non-unital associative algebra. The \emph{adic filtration} of $A$ is the descending filtration $A=A^1 \supseteq A^2 \supseteq \ldots$, where $A^n$ is the ideal generated by $n$-fold products of elements of $A$. The \emph{completion} of $A$ is $A^\wedge \coloneqq \varprojlim A/A^n$.  Similarly, if $\g$ is a Lie algebra, then its \emph{completion} is $\g^\wedge \coloneqq \varprojlim \g/L^n\g$, where $L^n\g$ is the $n$th term in the \emph{lower central series} of $\g$, i.e.\ $L^1\g=\g$ and $L^n \g = [\g,L^{n-1}\g]$ for $n>1$.
\end{defn}

\begin{defn}
	The \emph{homotopy completion} of an associative algebra or Lie algebra is the completion of any cofibrant replacement of the algebra.
\end{defn}

\begin{para}
	We will need some background on $\infty$-coalgebras and their $\infty$-morphisms.	For the moment we will just state the definition of the types of coalgebras we need in this section. The general discussion is deferred to \cref{sect:background}.
\end{para}

\begin{defn}\label{def Coo coalgebra}
	A \emph{$\Coo$-coalgebra} is a chain complex $C$ together with a continuous square-zero derivation of degree $-1$ on $\lie(s^{-1}C)^\wedge$ (the completion of the free Lie algebra on the desuspension of $C$) whose linear term vanishes. Similarly, an \emph{$\Aoo$-coalgebra} is a chain complex $C$ with a continuous square-zero derivation of degree $-1$ on the completion of the tensor algebra on the desuspension of $C$ with vanishing linear term.
\end{defn}

\begin{para}\label{failure of rectification}
	Recall that in Theorem \hyperref[thmA main]{A} we used three key properties (1), (2) and (3) of $\Coo$- and $\Aoo$-algebras, see \S\cref{properties}. In the coalgebra case, the first two properties remain true \emph{mutatis mutandis}, but the third property is problematic: if two cocommutative coalgebras are quasi-isomorphic as $\Coo$-coalgebras, then there is no reason for them to also be quasi-isomorphic as cocommutative coalgebras.\footnote{We do not actually have an example where property (3) fails, so the statement should be interpreted merely as saying that the usual proof of property (3) for $\infty$-algebras breaks down for $\infty$-coalgebras.} However, a weaker form of (3) remains true. In order to state it, we need to introduce the following notion of completed cobar constructions.
\end{para}

\begin{defn}Let $C$ be a $\Coo$-coalgebra. The \emph{completed cobar construction} $\CobarLC^\wedge(C)$ is the dg Lie algebra given by  $\lie(s^{-1}C)^\wedge$, together with the differential obtained from the square-zero derivation of \cref{def Coo coalgebra}. If $C$ is an $\Aoo$-coalgebra we define similarly $\CobarAC^\wedge(C)$.
\end{defn}
\begin{para}
	The name ``completed cobar construction'' for the functors $\CobarLC^\wedge$ and $\CobarAC^\wedge$ is natural, since if $C$ is a conilpotent cocommutative dg coalgebra (resp.\ conilpotent coassociative coalgebra), then one has
		\[
	\CobarLC^\wedge(C) = \CobarLC(C)^\wedge\qquad\text{and}\qquad\CobarAC^\wedge(C) = \CobarAC(C)^\wedge.
	\]
\end{para}
	
\begin{para}
	Recall that the bar constructions $\BarLC$ and $\BarAC$ preserve quasi-isomorphisms, but that the cobar constructions $\CobarLC$ and $\CobarAC$ do not. One way to understand this asymmetry is in terms of the natural filtrations on the various complexes involved (see \cref{sect:homotopy complete} for our conventions on filtered complexes). The functors $\BarLC$ and $\BarAC$ take quasi-isomorphisms of Lie, respectively associative algebras to filtered quasi-isomorphisms of coalgebras, where the coalgebras are filtered by the coradical filtration. A nearly identical argument shows similarly that $\CobarLC$ and $\CobarAC$ take quasi-isomorphisms of coalgebras to filtered quasi-isomorphisms 
	of Lie, resp. associative algebras, where Lie algebras are filtered by their lower central series, and associative algebras have the adic filtration. However, a filtered quasi-isomorphism of chain complexes will typically not be a quasi-isomorphism: this is in general only true if the filtrations are \emph{exhaustive} and \emph{complete}, see \cref{firstfiltrationlemma}. The reason that $\BarLC$ and $\BarAC$ preserve quasi-isomorphisms is that the coradical filtrations of the bar constructions are exhaustive and bounded below (in particular complete), so that a filtered quasi-isomorphism with respect to the coradical filtrations is a quasi-isomorphism. And the reason that $\CobarLC$ and $\CobarAC$ fail to preserve quasi-isomorphisms is that the lower central series filtration (resp.\ the adic filtration) on the cobar constructions are instead bounded above (in particular exhaustive) but \emph{not} complete.
\end{para}

\begin{para}
	The reasoning in the previous paragraph, however, makes it clear why the completed cobar constructions should preserve quasi-isomorphisms: the reason that $\CobarLC$ and $\CobarAC$ do not preserve quasi-isomorphisms was precisely that the lower central series, resp.\ adic filtrations of the cobar constructions are not complete. If we complete with respect to these filtrations we should obtain functors which \emph{do} preserve quasi-isomorphisms. We record this as a proposition for the moment; a more general statement will be proven as \cref{complete cobar preserves q-iso}. 
\end{para}

\begin{prop}
	Let $C \to D$ be a quasi-isomorphism of coassociative coalgebras.
	\begin{itemize}
		\item[i)] The induced map $\CobarAC^\wedge C \to \CobarAC^\wedge D$ is a quasi-isomorphism.
		\item[ii)] If $C$ and $D$ are moreover cocommutative, then the induced map $\CobarLC^\wedge C \to \CobarLC^\wedge D$ is a quasi-isomorphism.
	\end{itemize}
\end{prop}

\begin{para}\label{co-properties}
	We are now ready to state the properties of $\Coo$- and $\Aoo$-coalgebras we will need going forward. 
	\begin{enumerate}
		\item Any $\Coo$-coalgebra is $\Coo$-quasi-isomorphic to a \emph{minimal} $\Coo$-coalgebra, i.e.\ a $\Coo$-coalgebra with zero differential. (\cref{minimal coalgebras})
		\item If two $\Coo$-coalgebras $C$ and $D$ are quasi-isomorphic, then there exists a $\Coo$-quasi-isomorphism\footnote{As opposed to a zig-zag of quasi-isomorphisms.} $C \rightsquigarrow D$. (\cref{quasi-inverse coalgebras})
		\item If two cocommutative dg coalgebras $C$ and $D$ are quasi-isomorphic, then they are also $\Coo$-quasi-isomorphic. In the other direction, if $C$ and $D$ are $\Coo$-quasi-isomorphic, then $\CobarLC^\wedge C$ and $\CobarLC^\wedge D$ are quasi-isomorphic. (\cref{complete cobar preserves q-iso})
	\end{enumerate}
	Note that by the preceding proposition we may think of $\CobarLC^\wedge C \simeq \CobarLC^\wedge D$ as a weakened form of the statement $C \simeq D$. Similarly, $\Aoo$-coalgebras satisfy analogous properties (1'), (2') and (3'), where in (3') we use the completed cobar functor $\CobarAC^\wedge$. 
\end{para}

\begin{para}\label{codef cpx}
	We will also need the existence of \emph{deformation complexes} parametrizing the $\Aoo$- and $\Coo$-coalgebra structures on a given chain complex $V$. These are complete filtered graded dg Lie algebras whose Maurer--Cartan elements are $\Aoo$-coalgebra (resp.\ $\Coo$-coalgebra) structures on $V$, and whose gauge equivalences are $\Aoo$-isotopies (resp. $\Coo$-isotopies). They can be explicitly written as
	\[
	\coDef_{\Aoo}(V)\coloneqq\prod_{n \geq 2}\hom_{\mathbb{S}_n}\left(\susp\coass(n),\hom_\K(V,V^{\otimes n})\right)
	\]
	and
	\[
	\coDef_{\Coo}(V)\coloneqq\prod_{n \geq 2}\hom_{\mathbb{S}_n}\left(\susp\colie(n),\hom_\K(V,V^{\otimes n})\right),
	\]
	in complete analogy with the usual deformation complexes $\Def_{\Aoo}(V)$ and $\Def_{\Coo}(V)$. The fact that the Maurer--Cartan elements of these dg Lie algebras correspond to $\Aoo$- and $\Coo$-coalgebra structures on $V$ is equivalent to the statement that an $\Aoo$-coalgebra structure on $V$ is the same thing as an operadic twisting morphism from $\susp\coass$ to the coendomorphism operad $\coEnd_V$, and similarly a $\Coo$-coalgebra structure on $V$ is the same thing as an operadic twisting morphism from $\susp\colie$ to $\coEnd_V$. As for the usual deformation complex it is easy to identify the gauge equivalences between Maurer--Cartan elements with $\infty$-isotopies \emph{as a set}. What is less trivial is that this identification takes the composition of two gauge equivalences to the composition of the corresponding isotopies. This statement is the dual of \cite[Theorem 3]{preliedeformationtheory} and can be proven in exactly the same way.
\end{para}

\begin{para}
	We can now prove the following result, which is dual to Theorem \hyperref[thmA main]{A}.
\end{para}

\begin{thm}[Dual form of Theorem {\hyperref[thmA main]{A}}]
	\label{dualThmA} Let $C$ and $D$ be $\Coo$-coalgebras. Then $C$ and $D$ are quasi-isomorphic as $\Aoo$-coalgebras if and only if they are also quasi-isomorphic as $\Coo$-coalgebras.
\end{thm}

\begin{proof}
	By \S\cref{co-properties}(1) and (2') we may assume that $C$ and $D$ are minimal and that we have an $\Aoo$-quasi-isomorphism $f \colon C \rightsquigarrow D$. Since a quasi-isomorphism between chain complexes with vanishing differential is just an isomorphism, it follows that $f$ induces an isomorphism between the underlying graded vector spaces of $C$ and $D$. We can transport the $\Coo$-structure of one of the coalgebras along this isomorphism and reduce to the case where $C$ and $D$ are $\Coo$-coalgebras with the same underlying graded vector space $H$ that are linked by an $\Aoo$-morphism whose linear component is given by the identity, i.e.\ an $\Aoo$-isotopy.

	\medskip

	To finish the proof, we apply \cref{pbwcor} to the deformation complexes $\coDef_{\Aoo}(H)$ and $\coDef_{\Coo}(H)$ of $\Aoo$-coalgebra and $\Coo$-coalgebra structures on $H$. Indeed, by \S\ref{codef cpx} we now have two Maurer--Cartan elements of $\coDef_{\Coo}(H)$ which are gauge equivalent in the larger Lie algebra $\coDef_{\Aoo}(H)$. By applying \cref{pbwcor} in exactly the same way as in the proof of Theorem \hyperref[thmA main]{A} we see that there is a filtered retraction of $\coDef_{\Aoo}(H)$ onto $\coDef_{\Coo}(H)$, and applying \cref{retractionthm} as in the proof of Theorem \hyperref[thmA main]{A} yields the result. 
\end{proof}

\begin{para}
	Finally, we can put all of this together to conclude the proof of Theorem \hyperref[thm B stated again]{B}.
\end{para}

\begin{proof}[Proof of Theorem \text{\hyperref[thm B stated again]{B}}]
	Suppose that $\g$ and $\h$ are dg Lie algebras and that $U\g$ quasi-isomorphic to $U\h$. By \cref{Cg q-iso Ch}, we deduce that $\BarLC \g$ and $\BarLC \h$ are quasi-isomorphic as coassociative coalgebras, and in particular quasi-isomorphic as $\Aoo$-coalgebras. Therefore, by \cref{dualThmA} we have that $\BarLC \g$ and $\BarLC \h$ are also $\Coo$-quasi-isomorphic. By \S\ref{co-properties}(3) we deduce that $\CobarLC^\wedge \BarLC \g = (\CobarLC \BarLC \g)^\wedge$ and $\CobarLC^\wedge \BarLC \h = (\CobarLC \BarLC \h)^\wedge$ are quasi-isomorphic. But $\CobarLC \BarLC \g$ and $\CobarLC \BarLC \h$ are cofibrant replacements of $\g$ and $\h$, so saying that their completions are quasi-isomorphic to each other is exactly the same as saying that $\hc \g \simeq \hc \h$, concluding the proof. 
\end{proof}

\begin{rem}\label{remark: general koszul duality}
	There exists an analogue of Theorem \hyperref[thm B stated again]{B} for Lie coalgebras which is literally Koszul dual to Theorem \hyperref[thmA main]{A}, in the sense that the two statements are formally equivalent and interchanged by Koszul duality. Let $\P$ and $\Q$ be Koszul operads with Koszul dual cooperads $\P^\antishriek$ and $\Q^\antishriek$, and suppose we are given a commutative diagram 
	\[
	\begin{tikzcd}
		\P^\antishriek \arrow[r]\arrow[d]& \P\arrow[d] \\
		\Q^\antishriek \arrow[r]& \Q 
	\end{tikzcd}
	\]
	in which the horizontal arrows are the canonical Koszul twisting morphisms, and the vertical arrows are morphisms of cooperads and operads, respectively. There is then a commuting square of adjunctions
	\[
	\begin{tikzcd}
		\Picoalg \arrow[r,shift right]\arrow[d,shift right]& \Palg\arrow[d,shift right]\arrow[l,shift right] \\
		\Qicoalg \arrow[r,shift right]\arrow[u,shift right]& \Qalg \arrow[u,shift right]\arrow[l,shift right]
	\end{tikzcd}
	\]
	where $\Picoalg$ denotes the category of conilpotent $\P^\antishriek$-coalgebras, and similarly for $\Qicoalg$. The horizontal adjunctions are the bar-cobar adjunctions, the right vertical arrows are given by restriction and operadic pushforward, and the left vertical arrows are given by corestriction and cooperadic pullback. The fact that this square of adjunctions commutes merely means that the square of left adjoints (equivalently, the square of right adjoints) commutes, which is proven in much the same way as our \cref{lemma:isom cobar U}. The categories $\Palg$ and $\Qalg$ have model structures defined by Hinich \cite{hinichmodelstructure}, and these model structures may be transferred along the bar-cobar adjunctions to give model structures on conilpotent coalgebras for which the bar-cobar adjunctions are Quillen equivalences \cite[Theorem 2.1]{vallette}. In particular, a weak equivalence of coalgebras in this model structure is a morphism whose image under the cobar functor is a quasi-isomorphism. Under the resulting equivalences of categories
	\[
	\mathrm{Ho}(\Palg) \simeq \mathrm{Ho}(\Picoalg),\quad\mathrm{Ho}(\Qalg) \simeq \mathrm{Ho}(\Qicoalg)
	\]
	we obtain an identification of the restriction functor $\mathrm{Ho}(\Palg) \to \mathrm{Ho}(\Qalg)$ and the derived cooperadic pullback functor $\mathrm{Ho}(\Picoalg) \to \mathrm{Ho}(\Qicoalg)$. In this way any nontrivial theorem about the functor $\mathrm{Ho}(\Palg) \to \mathrm{Ho}(\Qalg)$ (e.g.\ that it reflects isomorphisms, as in our Theorem \hyperref[thmA main]{A}) can be equivalently restated as a theorem about coalgebras. We have for example the following:
	\begin{enumerate}[(i)]
		\item Theorem \hyperref[thmA main]{A} is equivalent to the statement that if $\mathfrak m$ and $\mathfrak n$ are conilpotent dg Lie coalgebras whose derived universal conilpotent coenveloping coalgebras are weakly equivalent, then $\mathfrak m$ and $\mathfrak n$ are themselves weakly equivalent. 
		\item The problem whether the universal enveloping algebra functor reflects quasi-isomorphisms is equivalent to the problem whether the forgetful functor from cocommutative conilpotent dg coalgebras to coassociative conilpotent dg coalgebras reflects weak equivalences.
	\end{enumerate}
	One might be tempted to try to modify our proof of Theorem \hyperref[thmA main]{A} to prove statement (ii) above. One stumbling block is that one would need an analogue of the deformation complex, which would be a dg Lie algebra whose Maurer--Cartan elements are conilpotent (or locally finite) $\infty$-coalgebra structures and whose gauges are locally finite $\infty$-isotopies which are moreover weak equivalences. It is far from clear {to us} that such an object even exists, given the indirect manner in which weak equivalences are defined.
\end{rem}

\section{Homotopy complete operadic algebras} \label{sect:homotopy complete}

\begin{para}
	The goal of this section is to prove that dg Lie algebras which are either non-negatively graded and nilpotent or negatively graded, are homotopy complete. We will in fact prove a result that applies to more general Koszul operads. Before going into the details, let us recall some well-known results about filtered complexes.
\end{para}

\subsec{Filtered complexes}

\begin{defn}
	A \emph{filtration} on a chain complex $V$ is a decreasing sequence of subspaces of $V$
	\[
	\cdots \supseteq \F^n V \supseteq \F^{n+1} V \supseteq \F^{n+2} V \supseteq \cdots
	\]
	such that $d(\F^n V) \subseteq \F^n V$  for all $n$. A morphism of filtered chain complexes $f:V \to W$ is a morphism of chain complexes which satisfies $f(\F^n V) \subseteq \F^n W$ for all $n$.
	The tensor product of two filtered chain complexes is itself a filtered chain complex via $F^n (V\otimes W) = \sum_{p+q=n} F^p V \otimes F^q W$.
\end{defn}

\begin{defn}\label{filtered-defn}
	Let $V$ be a filtered chain complex with filtration $\F^nV$. The filtration is called
	\begin{enumerate}
		\item \emph{exhaustive} if $\bigcup_n\F^nV = V$,
		\item \emph{complete} if the canonical map $V\to\varprojlim_n V/\F^nV$ is an isomorphism,
		\item \emph{bounded below}, respectively \emph{bounded above}, if for any homological degree $k$ there exists some $n$ such that $F^nV_k = 0$, respectively if $\F^nV_k = V_k$. 
	\end{enumerate}
	Note that bounded above implies exhaustive and bounded below implies complete. 
\end{defn}

\begin{rem}
	In our definition of bounded below and above we do \emph{not} insist that $F^nV = 0$ for some $n$, or that $F^nV=V$ for some $n$. These conditions are only imposed in each homological degree separately. This will be important e.g.\ when defining what it means for a dg Lie algebra to be nilpotent, in which case the correct condition is that the lower central series filtration is bounded below in the above sense.
\end{rem}

\begin{defn}\label{def filtered qi}
	Let $V$ be a filtered chain complex with filtration $\F^nV$. We denote $\Gr_\F^nV\coloneqq\F^nV/\F^{n+1}V$. A map of filtered chain complexes $V \to W$ is called a \emph{filtered quasi-isomorphism} if for all $n$ the induced map $\Gr_\F^n V \to \Gr_\F^n W$ is a quasi-isomorphism. 
\end{defn}

\begin{lem}\label{firstfiltrationlemma}
	Let $V \to W$ be a filtered quasi-isomorphism. Then the induced map \[
	\varprojlim_n \varinjlim_m F^mV/F^n V \to \varprojlim_n \varinjlim_m F^mW/F^n W
	\]
	is a quasi-isomorphism. In particular, a filtered quasi-isomorphism between exhaustive and complete filtered complexes is a quasi-isomorphism. 
\end{lem}

\begin{proof}
	We prove first that $F^mV/F^nV \to F^mV/F^nV$ is a quasi-isomorphism for any $n\ge m$. For this we fix an arbitrary $m$ and prove it for all $n \geq m$ by induction on $n$, using the short exact sequence
	\[
	\begin{tikzcd}
		0 \arrow[r]&\arrow[r]\arrow[d] \operatorname{Gr}_F^{n} V &\arrow[r]\arrow[d] F^{m} V/F^{n+1}V &\arrow[r]\arrow[d] F^{m} V/F^{n}V  & 0 \\
		0 \arrow[r]&\arrow[r] \operatorname{Gr}_F^{n} W & \arrow[r]F^{m} W/F^{n+1}W &\arrow[r] F^{m} W/F^{n}W & 0
	\end{tikzcd}
	\]
	and the five lemma. The result then follows from this, since $\varinjlim$ is an exact functor in general and $\varprojlim$ is exact when restricted to inverse systems of surjections. The latter fact can be proven either by a diagram chase, or by arguing that such an inverse system is fibrant for the injective model structure on diagrams, so that the limit of the diagram is a homotopy limit.
\end{proof}

\begin{para}
	Above we considered only \emph{decreasing} filtrations. We will denote increasing filtrations by a subscript, according to the convention $F^p V = F_{-p}V$. In this manner everything said above applies equally well to increasing filtrations. This convention is the exact analogue of using subscripts and superscripts to switch between homological and cohomological indexing.
\end{para}

\subsec{Homotopy complete algebras}

\begin{para}
	For the rest of this section, we let $\P$ be a Koszul operad concentrated in homological degree $0$. Then $\P^\antishriek(n)$ is concentrated in homological degree $n-1$ for all $n \in \mathbf N$. Examples are $\P = \lie, \com$ or $\ass$.
\end{para}

\begin{defn}\label{def:operadic filtration}
	Let $A$ be a $\P$-algebra. The \emph{operadic filtration} of $A$ is the  descending filtration defined by
	\[
	F^n A \coloneqq \mathrm{Image}(\P(n) \otimes A^{\otimes n} \to A).
	\]
\end{defn}

\begin{example}
	For $\P=\lie$ and $\P=\ass$, the operadic filtration specializes to the lower central series filtration and the adic filtration considered in \cref{adic and lcs filtration}, respectively.
\end{example}

\begin{defn}
	A $\P$-algebra is said to be \emph{nilpotent}\footnote{This notion is also sometimes referred to as \emph{degree-wise nilpotent} in the literature.} if its operadic filtration is bounded below. This is the case, for instance, if the algebra is concentrated in strictly positive degrees or strictly negative degrees.
\end{defn}

\begin{para}
	We will now define the homotopy completion, which depends on the choice of a cofibrant replacement functor. We let $\QQ A$ be the cofibrant replacement of $A$ given by the bar-cobar resolution. As a graded vector space, $\QQ A$ can be written explicitly as
	\begin{align*}
		\QQ A &= \bigoplus_{n \geq 1} (\P \circ \P^\antishriek)(n) \otimes_{\mathbb S_n} A^{\otimes n} \\
		&= \bigoplus_{n \geq 1} \P(n) \otimes_{\mathbb S_n} \bigoplus_{k_1+\ldots+k_n=\ell} \mathrm{Ind}_{\mathbb S_{k_1} \times \ldots \times \mathbb S_{k_n}}^{\mathbb S_\ell} \left( \P^\antishriek(k_1) \otimes \ldots \otimes \P^\antishriek(k_n) \right) \otimes_{\mathbb S_\ell} A^{\otimes \ell}.
	\end{align*}

\end{para}

\begin{defn}
	The \emph{completion} of $A$ is $A^\wedge = \varprojlim A/F^nA$. The \emph{homotopy completion} of $A$ is $\hc A \coloneqq (\QQ A)^\wedge$ \cite{harperhess}, and $A$ is said to be \emph{homotopy complete} if $\QQ A \to (\QQ A)^\wedge$ is a quasi-isomorphism. \label{def:homotopy completion}
\end{defn} 

\begin{lem}
	Suppose that $A$ is concentrated in {strictly} positive or negative degrees. Then $\QQ A$ has the same property.
\end{lem}

\begin{proof}We consider each summand in the explicit expression for $\QQ A$ separately. Suppose that $A$ is concentrated in negative degrees. Then $A^{\otimes \ell}$ is concentrated in degrees at most $-\ell$. If $k_1+\ldots+k_n=\ell$ then $\P^\antishriek(k_1) \otimes \ldots \otimes \P^\antishriek(k_n)$ has degree $\ell-n$. So their tensor product has degree at most $-n$. The case when $A$ is positively graded is obvious since both $A$ and $\P^\antishriek$ are concentrated in positive degrees. \end{proof}

\begin{prop}
	Let $A$ be {strictly} positively or negatively graded. Then $A$ is homotopy complete.
\end{prop}

\begin{proof}
	By the previous lemma, $\QQ A$ is nilpotent. Therefore, $\QQ A = (\QQ A)^\wedge$. 
\end{proof}

\begin{para}
	By a more careful argument we can also treat the case where $A$ is non-negatively graded and nilpotent. We will need to consider multiple distinct filtrations on $\QQ A$. We denote by $G$ the operadic filtration of $\QQ A$, which we can write explicitly as
	\[
	G^p \QQ A = \bigoplus_{n \geq p} \P(n) \otimes_{\mathbb S_n} \bigoplus_{k_1+\ldots+k_n=\ell} \mathrm{Ind}_{\mathbb S_{k_1} \times \ldots \times \mathbb S_{k_n}}^{\mathbb S_\ell} \left( \P^\antishriek(k_1) \otimes \ldots \otimes \P^\antishriek(k_n) \right) \otimes_{\mathbb S_\ell} A^{\otimes \ell}.
	\]
	The operadic filtration of $A$, which we denote by an $F$, induces a second filtration on $\QQ A$. Informally, we are just using the natural tensor product filtration on all the tensor powers $A^{\otimes n}$, and $\P$ and $\P^\antishriek$ are given the trivial filtration, so that
	\[
	F^p \QQ A = \bigoplus_{n \geq 1} (\P \circ \P^\antishriek)(n) \otimes_{\mathbb S_n} F^p (A^{\otimes n}).
	\]
	We refer to this filtration as the \emph{$F$-filtration}, and to the operadic filtration of $\QQ A$ as the \emph{$G$-filtration}. 
\end{para}

\begin{lem}\label{QA to A is a filtered qi}
	The map $\QQ A \to A$ is a filtered quasi-isomorphism with respect to the $F$-filtrations.
\end{lem}

\begin{proof}
	The proof is more or less identical to the usual proof that $\QQ A \to A$ is a quasi-isomorphism. Let us first recall this proof. We give $\QQ A$ a third filtration, which is the increasing filtration defined by
	\[
	L_p \QQ A \coloneqq \bigoplus_{n = 1}^p (\P \circ \P^\antishriek)(n) \otimes_{\mathbb S_n} A^{\otimes n}.
	\]
	We also define an increasing filtration on $A$ by $L_0A \coloneqq 0$ and $L_1 A \coloneqq A$. Since the $L$-filtrations are bounded below and exhaustive, it will be enough to show that $\QQ A \to A$ is a filtered quasi-isomorphism with respect to the $L$-filtrations. But we have
	\[
	\operatorname{Gr}^L_p \QQ A \cong (\P \circ \P^\antishriek)(p) \otimes_{\mathbb S_p} A^{\otimes p}.
	\]
	Now $(\P \circ \P^\antishriek)(p)$ is acyclic for $p>1$ since $\P$ is Koszul, and $(\P \circ \P^\antishriek)(1) \cong \K$. Therefore, the morphism $\operatorname{Gr}^L_p \QQ A \to \operatorname{Gr}^L_p A$ is a quasi-isomorphism for all $p$ as claimed.

	\medskip

	We now want to see that $\QQ A \to A$ is a filtered quasi-isomorphism with respect to the $F$-filtrations, i.e.\ that $\operatorname{Gr}_F^p \QQ A \to \operatorname{Gr}_F^p A$ is a quasi-isomorphism for all $p$. Again we use the $L$-filtration which is bounded below and exhaustive, so that it is enough that $\operatorname{Gr}^L_q \operatorname{Gr}_F^p \QQ A \to \operatorname{Gr}^L_q\operatorname{Gr}_F^p A$ is a quasi-isomorphism for all $p$ and $q$. But now we have
	\[
	\operatorname{Gr}^L_q\operatorname{Gr}_F^p A \cong
	\begin{cases}
		\operatorname{Gr}_F^p A & q = 1 \\ 0 & q \neq 1,
	\end{cases}
	\]
	and also $ \operatorname{Gr}^L_q \operatorname{Gr}_F^p\QQ A \cong (\P \circ \P^\antishriek)(q) \otimes_{\mathbb S_q} \operatorname{Gr}_F^p(A^{\otimes q}).$ 
	But by the same argument as in the previous paragraph this is  acyclic for $q \neq 1$ and isomorphic to $\operatorname{Gr}_F^p A$ when $q=1$, finishing the proof. 
\end{proof}

\begin{defn}
	Let $V$ be a chain complex with two filtrations $F$ and $G$. We say that $F$ and $G$ are \emph{commensurable} if for all homological degrees $k$ and for all $p$, there exists a $q$ such that $F^p V_k \supseteq G^q V_k$, and vice versa. 
	%.
\end{defn}

\begin{lem}
	Let $V$ be a chain complex with two commensurable filtrations $F$ and $G$. Then $\varprojlim V/F^nV$ is isomorphic to $\varprojlim V/G^nV$.
\end{lem}
\begin{proof}
	Clear by a standard cofinality argument. 
\end{proof}

%\begin{proof}
%	Fix a homological degree $k$. The assumptions imply that for all $p$ there exists an $q$ such that $F^p V_k \supseteq G^q V_k$, and vice versa. By a standard cofinality argument the two inverse limits agree.
%\end{proof}
\begin{lem}
	Suppose that $F$ and $G$ are bounded above filtrations on a chain complex $V$, such that the filtration on $\operatorname{Gr}^p_F(V)$ induced by $G$ is bounded below for all $p$, and the filtration on $\operatorname{Gr}^p_G(V)$ induced by $F$ is bounded below for all $p$. Then $F$ and $G$ are commensurable.
\end{lem}

\begin{proof}
	Fix a homological degree $k$. We need to show that for all $p$ there exists $q$ such that $F^p V_k \supseteq G^q V_k$. We prove this by induction on $p$, the base case being clear since $F^p V_k=V_k$ for $p\ll 0$. Now note that if $q$ is chosen large enough so that both $F^{p-1} V_k \supseteq G^q V_k$ (possible by induction) and also $G^q \operatorname{Gr}^p_F(V_k) = 0$ (possible by hypothesis), then $F^p V_k \supseteq G^q V_k$. 
\end{proof}

\begin{lem}
	Suppose that $A$ is concentrated in nonnegative degrees and nilpotent. Then the $F$-filtration and the $G$-filtration of $\QQ A$ are commensurable. 
\end{lem}

\begin{proof}
	Notice that since $A$ is nilpotent and non-negatively graded, the $F$-filtration on $A^{\otimes \ell}$ is bounded below for any $\ell$. Let us consider the expression
	\[
	\operatorname{Gr}_G^p \QQ A \cong \bigoplus_{\ell \geq p} \text{( some $\mathbb S_\ell$-representation concentrated in degree $\ell-p$ )} \otimes_{\mathbb S_\ell} A^{\otimes \ell}
	\]
	in a fixed homological degree $k$. Since $A$ is non-negatively graded there are in fact only finitely many terms in this direct sum which are nonzero in degree $k$, and for each of these finitely many terms the $F$-filtration is bounded below. 

	\medskip

	The other direction is true for any $\P$-algebra $A$, without any connectivity or nilpotence assumption. Indeed, consider the expression
	\[
	\operatorname{Gr}_F^p \QQ A \cong \bigoplus_{n \geq 1} \P(n) \otimes_{\mathbb S_n} \bigoplus_{k_1+\ldots+k_n=\ell} \mathrm{Ind}_{\mathbb S_{k_1} \times \ldots \times \mathbb S_{k_n}}^{\mathbb S_\ell} \left( \P^\antishriek(k_1) \otimes \ldots \otimes \P^\antishriek(k_n) \right) \otimes_{\mathbb S_\ell} \operatorname{Gr}_F^p A^{\otimes \ell}.
	\]
	Since $A = F^1A$ we must have $\operatorname{Gr}_F^p A^{\otimes \ell} = 0$ for $\ell > p$. Hence in the above sums all terms with $\ell > p$ vanish, so all terms with $n>p$ vanish, so $G^n \operatorname{Gr}_F^p \QQ A = 0$ for $n>p$. 
\end{proof}

\begin{thm}\label{completiontheorem}
	Suppose that $A$ is non-negatively graded and nilpotent. Then $A$ is homotopy complete.
\end{thm}

\begin{proof}
	Since $\QQ A \to A$ is a filtered quasi-isomorphism by \cref{QA to A is a filtered qi}, it induces a quasi-isomorphism between the completions with respect to the $F$-filtrations. Since $A$ is nilpotent this just means that the completion of $\QQ A$ with respect to the $F$-filtration is quasi-isomorphic to $A$. But by the previous lemma the completion of $\QQ A$ with respect to the $F$-filtration equals the completion with respect to the $G$-filtration, which means then that $A \simeq (\QQ A)^\wedge = \hc A$. 
\end{proof}

\begin{rem}
	In rational homotopy theory, positively graded dg Lie algebras over $\rats$ correspond to rational homotopy types of simply connected spaces (which was the case originally considered by Quillen \cite{quillenrationalhomotopytheory}), and non-negatively graded, nilpotent Lie algebras of finite type over $\rats$ correspond to connected, nilpotent spaces of finite $\rats$-type \cite{neisendorfer}. Thus, \cref{completiontheorem} essentially says that all dg Lie algebras arising from rational homotopy theory are homotopy complete.
\end{rem}

\section{A primer on infinity-coalgebras}\label{sect:background}

\begin{para}
	The goal of this section is to give a brief account of the formalism of $\infty$-coalgebras over a Koszul operad. We assume that the reader is comfortable with the corresponding formalism of $\infty$-algebras, see e.g.\ \cite[Chapter 10]{lodayvallette}. In particular we will try to clarify the difference between the various notions of $\infty$-coalgebra that one can find in the literature, and why it is important for us that we use precisely the definition we have chosen. This section will contain almost no proofs. We will several times refer to Hoffbeck--Leray--Vallette \cite{hoffbecklerayvallette} for theoretical results, although they work in a more general setting of ``gebras'' over a properad, meaning that they consider operations with multiple inputs and multiple outputs. In the special case of $\infty$-coalgebras these results were certainly known before \cite{hoffbecklerayvallette}, but we are not aware of a prior systematic treatment in the literature. 
\end{para}

\begin{para}
	Let us first recall that although one most often speaks of algebras over operads and coalgebras over cooperads, it is also possible to define the notions of a coalgebra over an operad and an algebra over a cooperad. If $\P$ is an operad, $\C$ is a cooperad, and $V$ is a chain complex, then we have the following schematic table:

	\begin{table}[h!]
		\begin{tabular}{lll}
 			$V$ an algebra over $\P$ &:  & maps $\P(n) \otimes_{\mathbb S_n} V^{\otimes n} \longrightarrow V$  \\
 			$V$ a coalgebra over $\C$&: & maps $V \longrightarrow  \C(n) \otimes_{\mathbb S_n} V^{\otimes n}$ \\
 			$V$ a coalgebra over $\P$&: & $\mathbb S_n$-equivariant maps $\P(n) \otimes V \longrightarrow  V^{\otimes n}$  \\
 			$V$ an algebra over $\C$&:  & $\mathbb S_n$-equivariant maps $V^{\otimes n} \longrightarrow  \C(n) \otimes V$.
		\end{tabular}
	\end{table}

	One can also think of a $\P$-coalgebra as a chain complex $V$ with a morphism of operads $\P \to \mathsf{coEnd}_V$ to the \emph{coendomorphism operad} $\mathsf{coEnd}_V$, with $\mathsf{coEnd}_V = \hom_\K(V,V^{\otimes n})$. If $\P(n)$ is finite dimensional for all $n$ then its linear dual $\P^\ast$ is a cooperad, and there are equivalences of categories between $\P$-algebras and $\P^\ast$-algebras, and between $\P$-coalgebras and $\P^\ast$-coalgebras. We caution the reader that a different definition of ``coalgebra over a cooperad'' appears in \cite{victor}. 
\end{para}

\begin{para}
	Let for the rest of this section $\P$ be a Koszul operad, with Koszul dual cooperad $\P^\antishriek$. We have $\P(0)=0$ and $\P(1) \cong \K$. We assume moreover that $\P(n)$ is finite dimensional for all $n$. We will now untangle the various definitions of a $\P_\infty$-coalgebra. 
\end{para}

\begin{para}
	Recall that $\P$ has a canonically defined \emph{Koszul resolution} $\P_\infty \coloneqq \Omega \P^\antishriek \to \P$, and that a $\P_\infty$-algebra is the same thing as an algebra over the Koszul resolution $\P_\infty$. Dualizing this definition leads to one possible definition of a $\P_\infty$-coalgebra, which is the one we are using in this article.
\end{para}

\begin{defn}\label{def Poo coalgebra}
	A \emph{$\P_\infty$-coalgebra} is a coalgebra over the operad $\P_\infty$. 
\end{defn}

\begin{rem}
	An equivalent definition of a $\P_\infty$-coalgebra structure on $C$ is an operadic twisting morphism from the Koszul dual cooperad $\P^\antishriek$ to $\mathsf{coEnd}_C$, cf. \cite[Theorem 6.5.7]{lodayvallette}. 
\end{rem}

\begin{para}
	One can also define the notion of a $\P_\infty$-algebra without mentioning the Koszul resolution or twisting morphisms, as follows. Let $\P^\antishriek(A)$ denote the cofree conilpotent $\P^\antishriek$-coalgebra cogenerated by a chain complex $A$. For example, if $\P=\com$ then $\P^\antishriek$ is the operadic suspension of the Lie cooperad, and $\P^\antishriek(A)$ is the cofree conilpotent Lie coalgebra cogenerated by $sA$. Then a $\P_\infty$-algebra structure on a chain complex $A$ is the same thing as a square-zero coderivation of degree $-1$ of $\P^\antishriek(A)$ whose linear term vanishes. Dualizing this definition leads to a \emph{different} notion of a $\P_\infty$-coalgebra than the one of \cref{def Poo coalgebra}.
\end{para}

\begin{defn}
	A \emph{locally finite $\P_\infty$-coalgebra} is a chain complex $C$ together with a square zero derivation of degree $-1$ with vanishing linear term on the free $\P^\antishriek$-algebra generated by $C$. 
\end{defn}

\begin{para}
	Unraveling all of the structures involved, one sees that a $\P_\infty$-coalgebra structure on a chain complex $V$ is given by a collection of $\mathbb S_n$-equivariant maps
	\[
	\P^\antishriek(n) \otimes C \to C^{\otimes n}
	\]
	satisfying an infinite sequence of quadratic equations, formally dual to those satisfied by the operations in a $\P_\infty$-algebra.  Taking the linear dual {of each $\P^\antishriek(n)$}, we may encode a $\P_\infty$-coalgebra by a single linear map
	\[
	C \to \prod_n \P^\antishriek(n)^\ast \otimes_{\mathbb S_n} C^{\otimes n}.
	\]
	On the other hand, a locally finite $\P_\infty$-coalgebra is given by a map
	\[
	C \to \bigoplus_n \P^\antishriek(n)^\ast \otimes_{\mathbb S_n} C^{\otimes n}
	\]
	(satisfying exactly the same quadratic equations as those of a $\P_\infty$-coalgebra), since the right hand side is the free $\P^\antishriek$-algebra generated by $C$, and a derivation of an algebra is determined by the values it takes on the generators. This shows that locally finite $\P_\infty$-coalgebras are exactly those $\P_\infty$-coalgebras for which the map $ C \to \prod_n \P^\antishriek(n)^\ast \otimes_{\mathbb S_n} C^{\otimes n}$ factors through the direct sum. 
\end{para}

\begin{para}
	We denote by $\P^\antishriek(C)^\wedge$ the completion of the free $\P^\antishriek$-algebra on $C$, cf.\ \cref{def:homotopy completion}. It can be written explicitly as	
	\[
	\P^\antishriek(C)^\wedge = \prod_n \P^\antishriek(n)^\ast \otimes_{\mathbb S_n} C^{\otimes n}.
	\]
	One has the following result, which is proven in exactly the same way as the analogous fact for $\P_\infty$-algebras \cite[Proposition 10.1.11]{lodayvallette}.
\end{para}

\begin{prop}
	There is a natural bijection between $\P_\infty$-coalgebra structures on a chain complex $C$ and (continuous) square zero derivations of degree $-1$ with vanishing linear term on \emph{$\P^\antishriek(C)^\wedge$}, the free complete \emph{$\P^\antishriek$}-algebra generated by $C$. 
\end{prop}

\begin{rem}
	In the previous proposition the word ``continuous'' is redundant. Indeed, since we are considering the operadic filtration, cf. Definition \ref{def:operadic filtration}, every derivation is necessarily filtration preserving and therefore continuous. Similarly, every morphism of $\P$-algebras is filtration preserving. In particular complete $\P$-algebras form a full subcategory of all $\P$-algebras.
\end{rem}

\begin{example}
	An $\Aoo$-coalgebra structure on $C$ can be identified with a square-zero derivation on the completion of the tensor algebra on $s^{-1}C$, and a $\Coo$-coalgebra structure can be identified with a square-zero derivation on the completion of the free Lie algebra on $s^{-1}C$. 
\end{example}

\begin{rem}
	We say that a $\P_\infty$-coalgebra $C$ is \emph{conilpotent} if there exists an exhaustive filtration of the form
	\[
	0 = F_0C \subseteq F_1C \subseteq F_2C \subseteq F_3C \subseteq \ldots
	\]
	such that all coalgebra structure maps preserve this filtration. A conilpotent $\P_\infty$-coalgebra is always locally finite, but not vice versa: there are strict inclusions
	\[
	(\text{conilpotent $\P_\infty$-coalgebras}) \subsetneq  (\text{locally finite $\P_\infty$-coalgebras}) \subsetneq  (\text{$\P_\infty$-coalgebras}).
	\]
	To see strictness of the first inclusion, note e.g.\ that any $\P$-coalgebra is also a locally finite $\P_\infty$-coalgebra, but not all $\P$-coalgebras are conilpotent. We caution the reader that all three notions are often referred to as simply ``$\P_\infty$-coalgebras'' in the literature.
\end{rem}

\begin{para}
	Let $C$ be a locally finite $\P_\infty$-coalgebra. Adding the square-zero derivation of $\P^\antishriek(C)$ to the internal differential of $\P^\antishriek(C)$ gives a differential graded $\P^\antishriek$-algebra that we call the \emph{cobar construction} on $C$, and that we denote $\Omega C$. Similarly if  $C$ is a general $\P_\infty$-coalgebra, then adding the square-zero derivation of $\P^\antishriek(C)^\wedge$ to the internal differential of $\P^\antishriek(C)^\wedge$ gives a differential graded $\P^\antishriek$-algebra that we call the \emph{completed cobar construction}, which we denote $\Omega^\wedge C$.
\end{para}

\begin{para}
	The definition of $\P_\infty$-coalgebra in terms of square-zero derivations is more convenient when one wants to define the notion of a morphism of $\P_\infty$-coalgebras. 
\end{para}

\begin{defn}
	Let $C$ and $D$ be $\P_\infty$-coalgebras. A \emph{$\P_\infty$-morphism} $C \rightsquigarrow D$ is a morphism of $\P^\antishriek$-algebras $\Omega^\wedge C \to \Omega^\wedge D$. If additionally $C$ and $D$ are locally finite, then a \emph{locally finite $\P_\infty$-morphism} $C \rightsquigarrow D$ is a morphism $\Omega C \to \Omega D$. There is an evident notion of composition of $\P_\infty$-morphisms and locally finite $\P_\infty$-morphisms making $\P_\infty$-coalgebras (resp.\ locally finite $\P_\infty$-coalgebras) into a category. 
\end{defn}

\begin{para}
	A $\P_\infty$-morphism of $\P_\infty$-coalgebras $f \colon C \rightsquigarrow D$, considered as a map $\Omega^\wedge C \to \Omega^\wedge D$, is completely determined by how it acts on generators. This implies that it can be written in terms of its components, which are $\mathbb S_n$-equivariant maps
	\[
	f_n \colon \P^\antishriek(n) \otimes C \to D^{\otimes n}
	\]
	for $n \geq 1$. In particular, since $\P^\antishriek(1)\cong\K$, the first component is a map $f_1 \colon C \to D$, which we call the \emph{linear term} of $f$.
\end{para}

\begin{defn}\label{properties of morphisms}
	A $\P_\infty$-morphism between $\P_\infty$-coalgebras is said to be a \emph{$\P_\infty$-quasi-isomorphism} if its linear term is a quasi-isomorphism and a \emph{$\P_\infty$-isomorphism} if its linear term is an isomorphism. A $\P_\infty$-morphism between two $\P_\infty$-coalgebras with the same underlying chain complex is called a \emph{$\P_\infty$-isotopy} if its linear term is the identity. 
\end{defn}

\begin{para}
	The definitions of $\P_\infty$-quasi-isomorphism and $\P_\infty$-isomorphism given in \cref{properties of morphisms} are justified by the following two facts:
	\begin{enumerate}
		\item If $C$ and $D$ are $\P_\infty$-coalgebras, then their homologies $H(C)$ and $H(D)$ are naturally $\P$-coalgebras. If $f \colon C \rightsquigarrow D$ is a $\P_\infty$-morphism, then $H(f_1) \colon H(C) \to H(D)$ is a morphism of $\P$-coalgebras. So $C \rightsquigarrow D$ is a $\P_\infty$-quasi-isomorphism if and only if the induced map of $\P$-coalgebras $H(C)\to H(D)$ is an isomorphism.
		\item A $\P_\infty$-morphism $f \colon C \rightsquigarrow D$ of $\P_\infty$-coalgebras is a $\P_\infty$-isomorphism in the sense of \cref{properties of morphisms} if and only if there exist a $\P_\infty$-morphism $g \colon D \rightsquigarrow C$ such that $g \circ f = \mathrm{id}_C$ and $f \circ g = \mathrm{id}_D$. 
	\end{enumerate}
	For the proof of (2), see \cite[Theorem 3.22]{hoffbecklerayvallette}; it is virtually identical to the proof of the analogous property of $\infty$-morphisms of $\P_\infty$-algebras. 
\end{para}

\begin{rem}\label{remark on inverse}
	Suppose that $f$ is a locally finite $\P_\infty$-morphism of locally finite $\P_\infty$-coalgebras. If $f$ is a $\P_\infty$-iso\-morphism in the sense of \cref{properties of morphisms}, then it is not necessarily an isomorphism in the category of  locally finite $\P_\infty$-coalgebras: its unique inverse in the category of all $\P_\infty$-coalgebras may not be a locally finite morphism. The reason is that if $f$ is given by a morphism $\Omega C \to \Omega D$ then the formula for its inverse is given by an infinite sum over trees, and this infinite sum will in general not converge unless the cobar constructions are completed. For a concrete example, consider the vector space $\K$ as an abelian $\Loo$-coalgebra, i.e.\ an $\Loo$-coalgebra with all cobrackets identically zero. Then the group of $\Loo$-isomorphisms $\K \rightsquigarrow \K$ is isomorphic to the group of formal power series over $\K$ in one variable with vanishing constant term and nonzero linear term, under composition. Such a power series corresponds to a locally finite $\Loo$-morphism $\K \rightsquigarrow \K$ if and only if it is a polynomial. Since the compositional inverse of a polynomial is in general only a power series, we see in particular that the inverse of a locally finite $\Loo$-morphism is in general not locally finite. 
\end{rem}

\begin{thm}\label{quasi-inverse coalgebras}
	Let $f \colon C \rightsquigarrow D$ be a quasi-isomorphism of $\P_\infty$-coalgebras. Then there exists a quasi-isomorphism $g \colon D \rightsquigarrow C$ such that the induced maps $H(C) \to H(D)$ and $H(D)\to H(C)$ are inverses. In particular, if two $\P_\infty$-coalgebras $C$ and $D$ are quasi-isomorphic, then there exists a $\P_\infty$-quasi-isomorphism $C \rightsquigarrow D$. 
\end{thm}

\begin{proof}
	See \cite[Theorem 4.18]{hoffbecklerayvallette}.
\end{proof}

\begin{defn}
	A $\P_\infty$-coalgebra is said to be \emph{minimal} if its underlying chain complex has vanishing differential.
\end{defn}

\begin{thm}\label{minimal coalgebras}
	Let $C$ be a $\P_\infty$-coalgebra. Then $C$ is quasi-isomorphic to a minimal $\P_\infty$-coalgebra, which is unique up to non-canonical $\P_\infty$-isomorphism.
\end{thm}

\begin{proof}
	This follows from a version of the Homotopy Transfer Theorem for $\infty$-coalgebras, see \cite[Theorem 4.14]{hoffbecklerayvallette}. By choosing a contraction from $C$ to $H(C)$ we can transfer the $\P_\infty$-coalgebra structure to a quasi-isomorphic structure on $H(C)$, which is then minimal. Uniqueness follows since a quasi-isomorphism between minimal $\P_\infty$-coalgebras is necessarily an isomorphism.
\end{proof}

\begin{para}
	Theorems \ref{quasi-inverse coalgebras} and \ref{minimal coalgebras} both rely on the Homotopy Transfer Theorem for $\infty$-coalgebras, which is proven using explicit ``sums over trees'' formulas. As in \cref{remark on inverse}, such an argument becomes problematic in the category of locally finite $\P_\infty$-coalgebras, since an infinite sum over trees will not have any reason to converge in that setting. We are not aware of any useful analogue of the Homotopy Transfer Theorem in the category of locally finite $\P_\infty$-coalgebras. Since both Theorems \ref{quasi-inverse coalgebras} and \ref{minimal coalgebras} are crucial for our arguments in Section \ref{sect:thmb} we are forced to work with in the category of general $\P_\infty$-coalgebras, even though all the $\P_\infty$-coalgebras we care about happen to be locally finite (in fact conilpotent). 
\end{para}

\begin{para}\label{par:cobar is not functorial}
	In \S\ref{failure of rectification} we mentioned that if two $\P$-algebras $A$ and $B$ are quasi-isomorphic in the category of $\P_\infty$-algebras, then they are also quasi-isomorphic in the category of $\P$-algebras, but that there is no reason for the corresponding statement for coalgebras to be true. Let us explain why this is the case, by first recalling how to prove the corresponding statement for $\P$-algebras. Suppose that $A$ and $B$ are $\P$-algebras, and that we have a $\P_\infty$-quasi-isomorphism $A \rightsquigarrow B$. Applying the bar and cobar functors gives a morphism of $\P$-algebras $\CobarAC \BarAC A \to \CobarAC \BarAC B$, and the counit of the bar-cobar adjunction furnishes morphisms of $\P$-algebras $\CobarAC \BarAC A \to A$, and $\CobarAC \BarAC B \to B$ which fit together in a commutative square
	\[
	\begin{tikzcd}
		\CobarAC \BarAC A \arrow[r]\arrow[d, "\sim" {anchor=south, rotate=90, inner sep=.5mm, xshift=.5mm}]& \CobarAC \BarAC B\arrow[d, "\sim" {anchor=south, rotate=90, inner sep=.5mm, xshift=.5mm}]\\
		A \arrow[r,rightsquigarrow]& B. 
	\end{tikzcd}
	\]
	Since the bottom arrow and the vertical arrows are quasi-isomorphisms, so is the top one. Thus, the algebras $A$ and $B$ are connected by a zig-zag of $\P$-algebra quasi-isomorphisms, as claimed. Now if we had two \emph{locally finite} $\P$-coalgebras $C$ and $D$, and a \emph{locally finite} $\P_\infty$-coalgebra quasi-isomorphism $C \rightsquigarrow D$, one could write down an analogous diagram
	\[
	\begin{tikzcd}
 		\BarAC \CobarAC C \arrow[r]& \BarAC \CobarAC D\\
		C \arrow[r,rightsquigarrow]\arrow[u, "\sim" {anchor=south, rotate=90, inner sep=.5mm, xshift=-.5mm}]& D  \arrow[u, "\sim" {anchor=south, rotate=90, inner sep=.5mm, xshift=-.5mm}]
	\end{tikzcd}
	\]
	and conclude by an identical argument that $C$ and $D$ are quasi-isomorphic as $\P$-coalgebras. But if $C$ and $D$ are not locally finite then $\Omega C$ and $\Omega D$ are undefined, and if the morphism $C \rightsquigarrow D$ is not locally finite then it does not correspond to a morphism $\Omega C \to \Omega D$. Either way, the above argument breaks down. 
\end{para}

\begin{para}
	One advantage of the completed cobar construction over the usual cobar construction is that the complete cobar construction preserves quasi-isomorphisms, unlike the usual one.
\end{para}

\begin{thm}\label{complete cobar preserves q-iso}
	Let $C \rightsquigarrow D$ be a $\P_\infty$-quasi-isomorphism of $\P_\infty$-coalgebras. Then the induced map $\Omega^\wedge C \to \Omega^\wedge D$ is a quasi-isomorphism. 
\end{thm}

\begin{proof}
	Recall that $\Omega^\wedge C = \prod_{n=1}^\infty \P^\antishriek(n) \otimes_{\mathbb S_n} C^{\otimes n}$. We introduce a descending filtration on $\Omega^\wedge C$ by the formula
	\[
	L^p \Omega^\wedge C = \prod_{n=p}^\infty  \P^\antishriek(n) \otimes_{\mathbb S_n} C^{\otimes n},
	\]
	and by the same formula we obtain a descending filtration on $\Omega^\wedge D$. The map $\Omega^\wedge C \to \Omega^\wedge D$ preserves filtrations and is a filtered quasi-isomorphism. Indeed, $\operatorname{Gr}_L^p \Omega^\wedge C \to \operatorname{Gr}_L^p \Omega^\wedge D$ is given by
	\[
	\begin{tikzcd}[column sep=4em] \P^\antishriek(p) \otimes_{\mathbb S_p} C^{\otimes p}\arrow[r,"\mathrm{id}\otimes f_1^{\otimes p}"] & \P^\antishriek(p) \otimes_{\mathbb S_p} D^{\otimes p} \end{tikzcd}
	\]
	where $f_1 \colon C \to D$ is the linear term of the $\infty$-morphism $C \rightsquigarrow D$, and by assumption $f_1$ is a quasi-isomorphism. The result follows from this since the filtrations are bounded above and complete.   
\end{proof}

\bibliographystyle{alpha}
\bibliography{Lie_associative_and_commutative_quasi-isomorphism}

\end{document}